\documentclass[11pt]{article} % use larger type; default would be 10pt

\usepackage[normalem]{ulem}

\usepackage{a4wide}
\usepackage{amsthm, amssymb, amsmath,hyperref,mathrsfs}
\usepackage{tikz-cd}
\usepackage{comment}
\newcounter{mnotecount}[section]

\renewcommand{\themnotecount}{\thesection.\arabic{mnotecount}}

\newcommand{\mnote}[1]%{}
{\protect{\stepcounter{mnotecount}}$^{\mbox{\footnotesize $%
\!\!\!\!\!\!\,\bullet$\themnotecount}}$ \marginpar{%\color{red}
\raggedright\tiny\em $\!\!\!\!\!\!\,\bullet$\themnotecount: #1} }

\newtheorem{Theorem}{Theorem}[section]
\newtheorem*{thm}{Theorem}
\newtheorem{Corollary}[Theorem]{Corollary}
\newtheorem{Lemma}[Theorem]{Lemma}
\newtheorem{Proposition}[Theorem]{Proposition}

\newtheorem{Definition}[Theorem]{Definition}

\newtheorem{rem}[Theorem]{Remark}
\numberwithin{equation}{section}

\DeclareMathOperator{\supp}{supp}
\DeclareMathOperator{\Char}{Char}

\newcommand{\R}{\mathbb{R}}
\newcommand{\N}{\mathbb{N}}
\newcommand{\tx}{\tilde{x}}
\newcommand{\ty}{\tilde{y}}
\newcommand{\txx}{{\bf{x}}}
\newcommand{\txixi}{{\boldsymbol{\xi}}}
\newcommand{\txi}{\tilde{\xi}}
\newcommand{\teta}{\tilde{\eta}}
\newcommand{\tep}{\tilde{\epsilon}}

\newcommand{\ang}[1]{\langle#1\rangle}

\newcommand{\Dp}[1]{{\cal D}'(#1)}

\usepackage{caption}
\tikzset{
  symbol/.style={
    draw=none,
    every to/.append style={
      edge node={node [sloped, allow upside down, auto=false]{$#1$}}}
  }
}

\title{The Sobolev Wavefront Set of the Causal Propagator in Finite Regularity}

\author{Yafet E. Sanchez Sanchez \and Elmar Schrohe} 

\newcommand{\Addresses}{{% additional braces for segregating \footnotesize
  \bigskip
  \footnotesize

 Y. Sanchez Sanchez , \textsc{INFN Sezione di Genova, Via Dodecaneso 33, 16146 Genova, Italy}\par\nopagebreak
  \textit{E-mail address}: \texttt{yafet.erasmo.sanchez.sanchez@edu.unige.it}\par\nopagebreak
  \textsc{Leibniz University Hannover, Institute of Analysis,  Welfengarten 1, 30167 Hannover, Germany}\par\nopagebreak
 \textit{E-mail address}: \texttt{yess@math.uni-hannover.de}

  \medskip
E. Schrohe, \textsc{Leibniz University Hannover, Institute of Analysis,  Welfengarten 1, 30167 Hannover, Germany}\par\nopagebreak
  \textit{E-mail address}: \texttt{schrohe@math.uni-hannover.de}

}}

\begin{document}
\maketitle
%-------------------------------------------------------------------------------------------------------------------------------------------------------
% ABSTRACT SECTION
%-------------------------------------------------------------------------------------------------------------------------------------------------------
\begin{abstract}
 Given a globally hyperbolic spacetime $M=\R\times \Sigma$ of dimension four and regularity $C^\tau$,
{
 we estimate the Sobolev wavefront set of the causal propagator $K_G$ of the Klein-Gordon operator. In the smooth case, the propagator satisfies $WF'(K_G)=C$, where $C\subset T^*(M\times M)$ { consists of those points $(\tx,\txi,\ty,\teta)$ such that $\txi,\tilde{\eta}$ are cotangent to a null geodesic $\gamma$ at $\tx$ 
resp. $\ty$ and parallel transports of each other along $\gamma$.}

We show that for  $\tau>2$,  $$WF'^{-2+\tau-{\epsilon}}(K_G)\subset C$$ for
every ${\epsilon}>0$. Furthermore, in regularity $C^{\tau+2}$ with $\tau>2$, 
    $$C\subset WF'^{-\frac{1}{2}}(K_G)\subset WF'^{\tau-\epsilon}(K_G)\subset C$$
holds for $0<\epsilon<\tau+\frac{1}{2}$.

In the ultrastatic case with $\Sigma$ compact, we show $WF'^{-\frac{3}{2}+\tau-\epsilon}(K_G)\subset C$ for $\epsilon >0$ and $\tau>2$ and  $WF'^{-\frac{3}{2}+\tau-\epsilon}(K_G)= C$ for $\tau>3$ and $\epsilon<\tau-3$. Moreover, we  show that the global regularity of the propagator $K_G$ is $H^{-\frac{1}{2}-\epsilon}_{loc}(M\times M)$ as  in the smooth case. 
%For the stationary case, we obtain the estimate $WF'^{-2+\tau-\epsilon}(K_G)\subset C$,% and for conformally-time-dependent stationary spacetimes, such as  Friedmann-Robertson-Walker-Lemaitre(FRWL) spacetimes,    $WF'^{-\epsilon}(K_G)\subset C$ for $\epsilon>0$ and $\tau>3$.
 }
\end{abstract}
\tableofcontents

Note: Theorem 3.4, Proposition 4.1, Corollary 4.4, Lemma 5.1 and Theorem 7.1 from Version 1 are now Theorem 3.3,  Proposition 6.7, Corollary 6.10, Lemma 4.1  and Theorem 6.19
respectively.

%-------------------------------------------------------------------------------------------------------------------------------------------------------
% INTRO SECTION
%-------------------------------------------------------------------------------
------------------------------------------------------------------------

\section{Introduction}
% in Curved Spacetime and in particular for the most spectacular prediction of that subject:  Black Hole evaporation.
\begin{comment}
\sout{While the main mathematical framework for the smooth setting was laid down already in the period 1980-2016 \cite{fulling, kaywald, rad, fewster,wrochna,moretti,hollands,stro, ko,fred,kav}, it is now possible to approach certain mathematical questions related to quantum fields propagating in spacetimes of low regularity. This is motivated by the deep foundational work on causality theory \cite{vien, sanchez, minguzzi, kunz} and to advances in our understanding of non-linear hyperbolic equations \cite{cr, dafermos, l2,luk} which were needed as a first step towards a full understanding of Einstein's equations. }
\end{comment}
{
The quantisation of the scalar field forms part of the basis for the subject of Algebraic Quantum Field Theory. While the main mathematical framework for the smooth setting was initiated more than 20 years ago, see e.g. \cite{haag, fulling, kaywald, rad, dim}, ongoing research continues to develop new techniques, particularly in connection with microlocal analysis \cite{junker, adiabatic, wrochna}, the importance of Hadamard states \cite{fewster, moretti, stro, ko, kav}, locality and covariance \cite{brun1, longo, few2}, perturbation theory  \cite{hollands, brun, buch}, Dirac fields \cite{holdir, gerdir, capo, drago}, and gauge theory \cite{ben, buch2}. 

Moreover, it is now possible to approach certain mathematical questions related to quantum fields propagating in spacetimes of finite regularity. This is motivated by the deep foundational work on causality theory \cite{vien, sanchez, minguzzi, kunz} and advances in our understanding of nonlinear hyperbolic equations \cite{cr, dafermos, l2}, which were needed as a first step towards a full understanding of Einstein's equations as a well-posed Cauchy problem, which requires solutions that go beyond the smooth ones. Additionally, there are several astrophysical models of phenomena such as neutron stars, self-gravitating fluids, and gravitational collapse that are not smooth \cite{adler,fluids,stars}.
}

The quantisation  proceeds in two steps. First, one constructs an algebra of observables, then  one represents this algebra on a Hilbert space of physical states.

A common candidate for such physical quantum states,  $\omega$,  are quasifree states that satisfy the microlocal spectrum condition.

To state it, it is useful to introduce the sets
\vspace{-.2cm}
\begin{eqnarray}\label{C}
  &C=\big\{(\tx,\txi,\ty,\tilde{\eta})\in T^{*}(M\times M)\backslash0;  
  g^{ab}(\tx)\txi_{a}{\txi}_{b}=g^{ab}(\ty)\tilde{\eta}_{a}\tilde{\eta}_{b}=0, (\tx,\txi)\sim(\ty,\tilde{\eta})\big\} \\
  &C^{+}=\left\{(\tx,\txi,\ty,\tilde{\eta})\in C; \txi^{0}\ge0,\tilde{\eta}^{0}\ge0\right\}\nonumber,
  \end{eqnarray} 
\noindent
where $(\tx,\txi)\sim(\ty,\tilde{\eta})$ means that there {is a null geodesic $\gamma$ joining $\tx$ and $\ty$ such that  $\txi,\tilde{\eta}$ are cotangent to the null geodesic $\gamma$ at $\tx$ 
resp. $\ty$ and parallel transports of each other.}
 
 Using the above sets one can define the microlocal spectrum condition as follows:
 \begin{Definition}
 A quasifree state $\omega_{H}$ on the algebra of observables satisfies the microlocal spectrum condition if its two point function $\omega^{(2)}_{H}$ is a distribution in $\mathcal{D}'(M\times M)$ and satisfies the following wavefront set condition
 
 $$WF'(\omega^{(2)}_{H})=C^{+},$$
 \noindent

where $WF'(\omega^{(2)}_{H}):= \{(\tx, \txi; \ty, -\tilde{\eta}) \in T^{*}(M\times M); (\tx, \txi; \ty, \tilde{\eta}) \in WF(\omega^{(2)}_{H})\}.$
 
 \end{Definition}

%\sout{These states are called Hadamard states \cite{rad, kaywald}.They include ground states and KMS states\cite{junker, wrochna}.}
 
{These states, called Hadamard states, 
have been constructed in the smooth setting. They encompass both ground and KMS states \cite{junker, wrochna}. Moreover, they are particularly well-suited for point-splitting renormalisation, a technique used for calculating key physical quantities like the renormalised energy-momentum tensor \cite{waldten, wald}.
}
\begin{comment}
 A larger class of states characterised in terms of their Sobolev wavefront set are the so-called adiabatic states of order $N$. These states are the natural generalisation of Hadamard states suitable for spacetimes of limited regularity. Moreover, adiabatic states of different order produce different physically measurable effects \cite{adiabatic}. 
 
 \begin{Definition}
 A quasifree state $\omega_{N}$ on the algebra  of observables
is called an adiabatic state of order $N\in\mathbb{R}$ if its two-point function $\omega^{(2)}_{N}$ is a
bidistribution that satisfies the following $H^s$-wavefront set condition for all $s\le N +\frac{3}{2}$
$$WF'^{s}(\omega^{(2)}_{N})\subset C^{+},$$
\noindent
 \end{Definition} 

\end{comment}

{
A central goal now is the construction of suitable quantum states in non-smooth scenarios following the techniques in \cite{adiabatic, wrochna}, which requires a thorough knowledge of the wavefront set of the causal propagator. This is the question we address in this article. To be precise, we characterise the wavefront set of the causal propagator of the Klein-Gordon operator in non-smooth globally hyperbolic spacetimes.}
%As a first result for the analysis of quasifree states in spacetimes of finite regularity we analyse the wavefront set of the causal propagator of the Klein-Gordon operator.
The causal propagator is constructed using the inverses associated with the Cauchy problem, which makes it a classical propagator.  It is worth noting that there exist other bisolutions such  as the two-point functions described above, which are non-classical (see \cite{der} for further details on this convention).

The microlocal analysis of the  propagators of the wave equation and its parametrices  in low regularity spacetimes introduces several technical challenges due to the lack of a complete theory of Fourier Integral Operators with non-smooth symbols and amplitudes.  However, progress has been made using the paradifferential calculus introduced by Bony \cite{bony} (see also \cite{beals,taylor,mar}).
% Here we mention some previous results.
In addition,  Szeftel has constructed  a  parametrix which requires only control over the $L^2$ curvature of the metric  in order to prove the $L^2$-curvature conjecture related to Einstein's field equations \cite{szeftelparametrix,l2}. Moreover, Tataru \cite{tataru} has constructed parametrices  of the wave equations  in low regularity for metrics with  $ C^{1,1}$  coefficients as a preliminary step to show suitable Strichartz estimates and analyse non-linear PDE's using phase space transforms. In addition, his results allowed even lower regularity at the expense of showing weaker results. Finally,  we mention Smith's construction of  parametrices for the $C^{1,1}$  case using wave packets \cite{smith} (see \cite{waters} for a parametrix construction using Gaussians).
The contribution of our  paper is establishing the microlocal singular structure of the causal propagator when the  regularity of the spacetime is finite. %\sout{We estimate the wavefront set of the causal propagator of the Klein-Gordon field for stationary spacetimes (Theorem \ref{mains}) and conformally stationary spacetimes (Theorem \ref{conformal}).  In the ultrastatic case, we show that better estimates are available, see Theorem \ref{main} and Theorem \ref{main2}. }}}
{
The main theorems we prove are:

\begin{thm}(Theorem \ref{mains})
Let $(M,g)$ be a $C^{\tau}$ globally hyperbolic spacetime with $\tau>2$ and $K_G$ the causal propagator of the Klein-Gordon operator $P$. Then $$WF'^{-2+\tau-{\epsilon}}(K_G)\subset C$$ for
every ${\epsilon}>0$, $C$ as in Eq.\eqref{C}, % when is wavefrontset empty?
\end{thm}
and
\begin{thm} (Theorem \ref{mainmain})
For a $C^{\tau+2}$ globally hyperbolic spacetime with $\tau>2$, 
    $$C\subset WF'^{-\frac{1}{2}}(K_G)\subset WF'^{\tau-\epsilon}(K_G)\subset C,$$
and hence equality, holds for $0<\epsilon<\tau+\frac{1}{2}$.
\end{thm}

In the ultrastatic case, sharper results are available. For completeness, we state these in the Appendix, see Lemma 6.5, Lemma 6.7, Theorem 6.9 and Theorem 6.11.

}

\subsection{The Smooth Setting}

Consider a pair  $(M,g)$, where $M$ is a smooth manifold  and $g$ is a  
smooth Lorentzian metric. The  Klein-Gordon operator $P$ on $(M,g)$ is given by
\begin{equation} \label{wave2}
P:=g^{\mu\nu}\nabla_{\mu}\nabla_{\nu}\phi+m^{2}\phi=(\square_g+m^2)\phi
\end{equation}
\noindent
where $g^{\mu\nu}$ is the inverse metric tensor, $\nabla_{\mu}$ is the covariant derivative and $m$ is a positive real number.

The starting point is the notion of advanced and retarded Green operators in this situation.

\begin{Definition}
\label{propagators}
  Let $M$ be a time-oriented connected Lorentzian manifold and let $P$
  be the Klein-Gordon operator. An \emph{advanced Green operator}
  $G^+$ is a linear map $G^+:{\mathcal D}(M) \to C^\infty(M)$ such
  that
\begin{enumerate}\renewcommand{\labelenumi}{\rm \arabic{enumi}.}
\item $P \circ G^+ ={\rm id}_{{\mathcal D}(M)}$
\item $G^+ \circ P|_{{\mathcal D}(M)} ={\rm id}_{{\mathcal D}(M)}$
\item ${\supp}(G^+\phi) \subset J^+({\supp}(\phi))$ 
for all $\phi \in {\mathcal D}(M)$.
\end{enumerate}
A \emph{retarded Green operator} $G^-$ satisfies $(1)$ and $(2)$, but $(3)$
is replaced by the condition  ${\supp}(G^-\phi) \subset J^-({\supp}(\phi))$ for all $\phi
\in {\mathcal D}(M)$.
\end{Definition}
In \cite[Corollary 3.4.3]{bgp} it is shown that these exist and
are unique on a globally hyperbolic manifold.

The advanced and retarded Green operators are then used to define the
causal propagator $$G:=G^+-G^-$$ which maps ${{\mathcal D}(M)}$ to $C^\infty_{{\rm sc}}(M)$, the space
of spatially compact maps , i.e.\ the smooth
maps $\phi$ such that there exists a compact subset $K \subset M$ with
$\supp(\phi) \subset J(K)$. If $M$ is globally hyperbolic then one has
the following \emph{exact sequence} \cite[Theorem 3.4.7]{bgp}:
\begin{equation*}
\begin{tikzcd}
0 \arrow{r}  &{\mathcal D}(M) \arrow{r}{P} &{\mathcal D}(M) 
\arrow{r}{G} &C^\infty_{{\rm sc}}(M) \arrow{r}{P} &C^\infty_{{\rm sc}}(M),
\end{tikzcd}
\end{equation*}
Since $G$ is a continuous linear operator, 
%from ${\mathcal D}(M)$ to ${\mathcal D}'(M)$. Therefore,
the Schwartz Kernel Theorem implies that there exists one and only one distribution ${\displaystyle K_G\in {\mathcal {D}}'(M\times M)}$ such that 
\begin{equation}
K_G(u\otimes v)=\ang{G(v),u}, \quad u,v\in \mathcal{D}(M).
\end{equation}
%
%holds for every $u,v\in {\mathcal D}(M)$.
It follows from Duistermaat and H{\"o}rmander's characterisation using Fourier Integral Operators that the kernel $K_G$ satisfies 
\begin{equation}
WF'(K_G)=C.
\end{equation}
 
More explicitly, they showed that $K_G\in I^{-\frac{3}{2}}(M\times M, C')$, where $I^{\mu}(X,\Lambda)$  denotes the space of Lagrangian distributions of order $\mu$ over the manifold $X$ associated to the Lagrangian submanifold $\Lambda$. In this case $\Lambda=C'=\{(\tx, \txi; \ty, -\tilde{\eta}) ;(\tx, \txi; \ty, \tilde{\eta}) \in C\}$, see \cite[Theorem 6.5.3]{duistermaat}.{  Using \cite[Theorem 5.4.1, Theorem 6.5.3]{duistermaat}  one obtains that in four dimensions, {$K_G$ belongs to the Sobolev space $H_{loc}^{-\frac{1}{2}-\epsilon}(M\times M)$ for any $\epsilon>0$.} {For details on the Sobolev spaces mentioned, see Section \ref{AppendixA} and \cite[Appendix B]{green}. }

\section{The Non-smooth Setting}

Next  we will consider the case, where $g$ {is a non-smooth metric. We will specify the precise regularity in each section. }

The definition of the Green operators in the non-smooth setting will
require us to choose suitable spaces of functions based on Sobolev spaces as domain and range. We let %therefore define the following spaces:
\begin{align}
V_0&= \{\phi\in H_{{comp}}^{2}(M);  %\text{ s.t. } 
P\phi\in  H^{1}_{{comp}}(M) \} 
\nonumber\\
V_{sc}&= \{\phi\in  H_{{loc}}^{2}(M); %\text{ s.t. } 
P\phi\in H^{1}_{{loc}}(M) \\
&\text{ and } \text{supp}(\phi) \subset J(K), \text{ where }K \text{ is a compact subset of } M\}.\nonumber
\end{align}

\begin{Definition} \label{Green} 
An \emph{advanced  Green operator} for the Klein-Gordon operator $P$ is a
linear map $$G^{+}:H_{\text{comp}}^{1}(M)\rightarrow H_{{loc}}^{2}(M)$$ satisfying the  properties
 \begin{enumerate}\renewcommand{\labelenumi}{\rm \arabic{enumi}.}
   \item $PG^{+}=\text{id}_{H_{{comp}}^{1}(M)}$, 
   \item $ G^{+}P|_{V_0}=\text{id}_{V_0}$,
   \item $\supp(G^{+}(f))\subset J^{+}(\supp(f))$ for all $f\in H_{{comp}}^{1}(M)$,
 \end{enumerate}
 A \emph{retarded 
 Green operator} $G^{-}$  is defined correspondingly.
\end{Definition}

It was shown in \cite[Theorem 5.8]{green} that these  operators exist and are unique on Lorentzian manifolds that satisfy the condition of generalised hyperbolicity. This condition is satisfied in particular for $C^{1,1}$ globally hyperbolic spacetimes.
Moreover, one obtains a short exact sequence for  the low-regularity causal propagator, $G:=G^+-G^-$,  similar to  that in the smooth case % \cite{green}
\begin{center}
\begin{tikzcd}
    0\arrow{r}& V_0 \arrow{r}{P}& {H_{{comp}}^{1}(M)}\arrow{r}{G} &V_{\text{sc}} \arrow{r}{P} & H^{1}_{{loc}} (M).
   \end{tikzcd}
\end{center}

\section{Pseudodifferential Operators with Non-smooth  Symbols}

\subsection{Symbol Classes}

Let $\{\psi_j; j=0,1,\ldots\}$ be a Littlewood-Paley partition of unity on 
$\mathbb R^n$, i.e., a partition of unity $1=\sum_{j=0}^\infty\psi_j$, 
where $\psi_0\equiv 1$ for $|\xi|\le 1$ and $\psi_0\equiv 0$ for $|\xi|\ge2$ and 
$\psi_j(\xi) =  \psi_0(2^j\xi)-\psi_0(2^{1-j}\xi)$. 
The support of $\psi_j$, $j\ge1$, then lies in an annulus around the origin of interior radius 
$2^j$ and exterior radius $2^{1+j}$.  
%$\supp(\psi_j)\sim \langle\xi\rangle\sim 2^j$ and $\psi_j(\xi)=\psi_1(2^{1-j}\xi)$ for $j\ge2$.

\begin{Definition}\label{Holder}
{\rm (a)} 
For $\tau\in (0,\infty)$, the H{\"o}lder space $C^\tau(\mathbb{R}^n)$ is the set of all functions $f$ with
\begin{equation}
\|f\|_{C^\tau}:=\displaystyle\sum_{|\alpha|\le [\tau]}\|\partial^\alpha_{x}f\|_{L^\infty(\mathbb{R}^n)}+\displaystyle\sum_{|\alpha|= [\tau]}\sup_{x\neq y}\frac{|\partial^\alpha_{x}f(x)-\partial^\alpha_{x}f(y)|}{|x-y|^{\tau-[\tau]}}<\infty.
\end{equation}

{\rm (b)} For $\tau\in \mathbb{R}$ the Zygmund space 
$C^\tau_{*}(\mathbb{R}^n)$ consists of all functions $f$ with 
\begin{equation}
\|f\|_{C^\tau_*}=\sup_j 2^{j\tau}\|\psi_j (D) f\|_{L^\infty}<\infty.
\end{equation}
\end{Definition}
Here $\psi_j(D)$ is the Fourier multiplier with symbol $\psi_j$, i.e., 
$\psi_j(D)u = \mathcal F^{-1}\psi_j\mathcal F u$, 
where $(\mathcal Fu)(\xi) = (2\pi)^{-n/2} \int e^{-ix\xi} u(x)\, d^nx$ is the Fourier transform. 

We have the following relations: $C^\tau=C^\tau_{*}$ if $\tau\notin \mathbb{N}$, and $C^\tau\subset C^\tau_{*}$ if $\tau\in \mathbb{N}$.

We next introduce  symbol classes % of symbols $p(x,\xi)$ in $\mathbb{R}^n
of finite H\"older or Zygmund regularity{,} following Taylor % regularity in $x$. We follow 
\cite{taylor}. We use the notation $\langle\xi\rangle:=(1+|\xi|^2)^{\frac{1}{2}}$, $\xi\in \mathbb R^n$.

\begin{Definition}{\rm (a)} Let $0\le \delta <1$.
A symbol $p(x,\xi)$ belongs to $C^\tau_* S^{m}_{1,\delta}:=C^\tau_*S^m_{1,\delta}(\R^n\times \R^n)$ if  $$\|D^{\alpha}_{\xi}p(\cdot,\xi)\|_{C_*^{\tau}}\le C_\alpha\langle\xi\rangle^{m-|\alpha|+\tau\delta} \text{ and } |D^{\alpha}_{\xi}p(x,\xi)|\le C_\alpha\langle\xi\rangle^{m-|\alpha|}.$$

{\rm (b)}  We obtain the symbol class $C^\tau S^{m}_{1,\delta}:=C^\tau S^m_{1,\delta}(\R^n\times \R^n)$ for $\tau>0$ by requiring that
%replacing $C_*^\tau$ by $C^\tau$ in the above definition and requiring additionally that 
$$\|D^{\alpha}_{\xi}p(\cdot,\xi)\|_{C^{s}}\le C_\alpha\langle\xi\rangle^{m-|\alpha|+s\delta}, \quad0\le s\le \tau$$.%, \text{ and } |D^{\alpha}_{\xi}p(x,\xi)|\le C_\alpha\langle\xi\rangle^{m-|\alpha|}.$$
%. \yafet{Is this equivalent to Taylors definition in 1.3.18 in non linear PDE?}

{\rm (c)} A symbol $p(x,\xi)$ is in  $C^{\tau}S_{cl}^{m}$ provided $p(x,\xi)\in C^{\tau}S^{m}_{1,0}$ and $p(x,\xi)$ has a classical expansion 
$$p(x,\xi)\sim \sum_{j\ge0}p_{m-j}(x,\xi)$$ 
in terms $p_{m-j}$ homogeneous of degree $m-j$ in $\xi$ for $|\xi|\ge 1$, in the sense that the difference between $p(x,\xi)$ and the sum over $0\le j< N$ belongs to $C^{\tau}S^{m-N}_{1,0}$.

\end{Definition}

%\begin{rem}
%The Klein-Gordon operator $P$ given by Eq.\eqref{kg}  has the symbol $P(\tilde{x},\tilde{\xi})=(-\xi_0^2+h^{ij}\xi_i\xi_j)+i\frac{1}{\sqrt{h}}\partial_{x^i}(h^{ij}\sqrt{h})\xi_j+m^2$. 
%For a metric of regularity $C^{\tau}$, the symbol $P(\tilde{x},\tilde{\xi})$ belongs to $C^{\tau-1}S^{2}_{cl}$.
%\end{rem}

The pseudodifferential operator $p(x,D_x)$ with the symbol $p(x,\xi)\in C^\tau S^m_{1,\delta} $ is given by 
\begin{equation}
\left(p(x,D_x)u\right)(x)=(2\pi)^{-n/2} \int_{\mathbb{R}^n}e^{ix\cdot\xi}p(x,\xi)({\cal{F}}{u})(\xi)d^n\xi,
\quad u\in \mathcal S(\mathbb R^n). 
\end{equation}
It extends to continuous maps 
\begin{equation}\label{continuous}
p(x,D_x): H^{s+m}(\mathbb{R}^n)\rightarrow H^{s}(\mathbb{R}^n) , \quad -\tau(1-\delta)<s<\tau.
\end{equation}

{
While it is possible to extend the theory of pseudodifferential operators with non-smooth symbols to manifolds (see \cite{carolina}), due to the local nature of our results it is a key point of this article that we can work entirely on $\R^n$.}

\subsection{Symbol Smoothing} Given $p(x,\xi)\in C^{\tau}S_{1,\gamma}^{m}$ and $\delta\in (\gamma,1)$ let 
\begin{equation}
p^{\#}(x,\xi)=\displaystyle\sum_{j=0}^{\infty}J_{\epsilon_{j}}p(x,\xi)\psi_j(\xi).
\end{equation}
Here $J_\epsilon$ is the smoothing operator given by 
$(J_\epsilon f)(x)=(\phi(\epsilon D)f)(x)$ with $\phi\in C^{\infty}_0(\mathbb{R}^n)$, $\phi(\xi)=1$ for $|\xi|\le 1$, and we take $\epsilon_j=2^{-j(\delta-\gamma)}$.

Letting $p^{b}(x,\xi)=p(x,\xi)-p^{\#}(x,\xi)$ we obtain the decomposition
\begin{equation}\label{smoothing}
p(x,\xi)=p^{\#}(x,\xi)+p^{b}(x,\xi),
\end{equation}
where $p^{\#}(x,\xi)\in S^{m}_{1,\delta}$ and $p^{b}(x,\xi)\in C^{\tau}S^{m-\tau(\delta-\gamma)}_{1,\delta}$. 

{The symbol estimates for $p^\#$ are a consequence of the estimate $$\|\partial_x^\beta J_\epsilon f\|_{L^\infty}\le \begin{cases}
    C\|f\|_{C^\tau} \quad |\beta|\le \tau\\
    C\epsilon^{-(|\beta|-\tau)}\|f\|_{C^\tau} \quad |\beta|> \tau,
\end{cases}$$

and that $\epsilon_j=2^{-j(\delta-\gamma)}$. For details see Proposition 1.3 E and Equation (1.3.21) in \cite{taylor}.}

\subsection{Microlocal Sobolev Regularity}

%\begin{Definition}
Let $p\in C^\tau S^{m}_{\rho,\delta}$, $\tau>0$,  with $\delta<\rho$. Suppose that there is a conic neighborhood $\Gamma$ of $(x_0,\xi_0)$ and constants $c,C>0$ such that 
$|p(x,\xi)|\ge c|\xi|^m$ for $(x,\xi) \in \Gamma$, $|\xi|\ge C$. Then $(x_0,\xi_0)$ is called {\em non-characteristic}.  If $p$ has a homogeneous principal symbol $p_m$, the condition is equivalent to $p_m(x_0,\xi_0)\neq 0$. The complement of the set of non-characteristic points  is  the set of characteristic points denoted by ${\Char}(p)$.
%\end{Definition}

A distribution $u$ is {\em microlocally in} $H^{s}$ at $(x_0,\xi_0)\in T^*M\backslash 0$ if 
there exists a conic neighbourhood ${\Gamma_0}$ of $\xi_0$ and a smooth function $\varphi\in C_{0}^{\infty}(M)$ with $\varphi (x_0)\neq 0$ 
such that $$\int_{{\Gamma_0}}\ang{\xi}^{2s}|{\cal{F}}(\varphi u)(\xi)|^{2}d^{n}\xi <\infty.$$
Otherwise we say that $(x_0,\xi_0)$ lies in the $H^s$-wavefront set $WF^s(u)$.  

If $u$ is microlocally in $H^{s}$ in an open conic subset ${\Gamma}\subset T^*M\backslash 0$ we write $u\in H^s_{mcl}({\Gamma})$.

%If $u$ is not  microlocally $H^{s}$ at $(p,\chi)$, then we say $(p,\chi)\in WF^s (u)$.

\subsection{Propagation of Singularities for Bisolutions of the Klein-Gordon Operator}

A globally hyperbolic spacetime is of the form $\mathbb R\times \Sigma$, {where $\Sigma$  is not assumed to be compact,} and we will write local coordinates in the form  
\begin{eqnarray}\label{xcoord}%\lefteqn{}\\
\tx = (t,x), \ty=(s,y)
\end{eqnarray}
and the associated covariables as 
\begin{eqnarray}\label{xicoord}%\lefteqn{}\\
\txi = (\xi^0,\xi), \teta= (\eta^{0},\eta). % \txx= (\tx,\ty), \txixi=(\txi,\teta)
\end{eqnarray}
On the product $(\mathbb R\times \Sigma)\times( \mathbb R\times \Sigma)$ we use $(\txx,\txixi)$ with 
\begin{eqnarray}\label{doublecoord}%\lefteqn{}\\
\txx= (\tx,\ty), \txixi=(\txi,\teta).
\end{eqnarray}

In the sequel we shall apply the Klein-Gordon operator also to functions and distributions on $M\times M$. Using the coordinates in Eqs. \eqref{xcoord},\eqref{xicoord} and \eqref{doublecoord}, we distinguish  the cases, where $P$ acts on the first set of variables $(t, x)$  or on the second {set} $(s, y)$, and write $P_{(t,x)}$ and $P_{(s,y)}$, respectively. 
Explicitly,

\begin{eqnarray}%\label{}\lefteqn{}\\
\nonumber
P_{(t,x)}(\txx,D_{\txx})&=&P_{(t,x)}(\tx,D_{\tx},\ty,D_{\ty})=(\square_{g(\tx)}+m^2)\otimes I\\ %({-\xi^{0}}^{2}+h^{ij}(x)\xi_i\xi_j)+{i\frac{1}{\sqrt{h}}\partial_{x^i}(h^{ij}\sqrt{h}(x))\xi_j}+{m^2}\\
\label{pr0}\nonumber
P_{(s,y)}(\txx,D_{\txx})&=&P_{(s,y)}(\tx,D_{\tx},\ty,D_{\ty})=I\otimes(\square_{g(\ty)}+m^2)
\end{eqnarray}

%(-{\eta^{0}}^{2}+h^{ij}(y)\eta_i\eta_j)+{i\frac{1}{\sqrt{h}}\partial_{y^i}(h^{ij}\sqrt{h}(y))\eta_j}+{m^2}.
In particular, 
\begin{eqnarray}%\label{}\lefteqn{}
\label{char5}
  \Char(P_{(t,x)})&=&\Char(P)\times T^{*}M\cup\{(\txx,\txixi)\in T^{*}(M\times M)\backslash0, \txi=0\}\\
\Char(P_{(s,y)})&=&T^{*}M\times \Char (P)\cup\{(\txx,\txixi)\in T^{*}(M\times M)\backslash0, \teta=0\}.\nonumber
\end{eqnarray}
%\end{equation}
%
%and $P_{(s,y)}(\txx,\txixi)$ given by
%
%\begin{equation}\label{pr0}
%P_{(s,y)}(\txx,\txixi)=:P_{(s,y)}(\tx,\txi,\ty,\teta)=\underbrace{(-{\eta^{0}}^{2}+h^{ij}(x)\eta_i\eta_j)}_{p_0(\txx,\txixi)}+\underbrace{i\frac{1}{\sqrt{h}}\partial_{y^i}(h^{ij}\sqrt{h}(y))\eta_j}_{p_1(\txx,\txixi)}+\underbrace{m^2}_{p_2(\txx,\txixi)}.
%\end{equation}

{
\begin{Theorem}\label{elliptic}Let the metric $g$ be of class $C^\tau$, $\tau>1$,
$0\le \sigma<\tau-1$ and $v\in H_{loc}^{2+\sigma-\tau+\epsilon}(M\times M)$ for some ${\epsilon}>0$ with 
 $P_{(t,x)}(\txx,D_\txx){v}=0$. Then 
% , where the metric satisfies $g\in C^\tau$ for $\tau>2$ , $v\in H_{loc}^{2+\sigma-\tau\delta}(M\times M)$. Then for $1>\delta>0, -(1-\delta)(\tau-1)<\sigma<\tau-1$
$$WF^{\sigma+2}({v})\subset \Char(P_{(t,x)}).$$
\end{Theorem}
}
\begin{proof}
Being interested in the wavefront set  of $v$ near a point $\txx$ we multiply $v$ by a function $\varphi\in \mathcal D(M\times M)$ with $\varphi \equiv 1$ near $\txx$ and consider $\varphi v$. So we can assume that $v$ has support in a small neighbourhood of $\txx$ contained in a single coordinate patch and consider $v$ as an element of $H^{2+\sigma-\tau+\epsilon}(\R^4\times \R^4)$. In order to distinguish points $(\txx,\txixi)=(\tx,\txi,\ty,\teta)$ from their representation in local coordinates, we will write the latter in the form $(\uline{\txx},\uline{\txixi})=(\uline{\tx}, \uline{\txi}, \uline{\ty},\uline{\teta})$. In this local setting, $P_{(t,x)}(\uline{\txx},D_{\uline{\txx}})$ is given by the symbol
%Let 
\begin{equation}\label{p0}
P_{(t,x)}(\uline{\txx},\uline{\txixi})=P_{(t,x)}(\uline{\tx},\uline{\txi},\uline{\ty},\uline{\teta})=\underbrace{g^{\mu\nu}(\uline{\tx})\uline{\xi_\mu}\uline{\xi_\nu}}_{p_2(\uline{\txx},\uline{\txixi})}+\underbrace{ig^{\mu\nu}(\uline{x})\Gamma^\rho_{\mu\nu}(\uline{x})\uline{\xi}_\rho}_{p_1(\uline{\txx},\uline{\txixi})}+\underbrace{m^2}_{p_0(\uline{\txx},\uline{\txixi})}.
\end{equation}

\begin{comment}

The symbol smoothing (Eq.(\ref{smoothing})) on $p_2,p_1$ gives a decomposition 
%
\begin{align*}
p_2(\txx,\txixi)&=p_2^{\#}(\txx,\txixi)+p_2^{b}(\txx,\txixi)\\
p_1(\txx,\txixi)&=p_1^{\#}(\txx,\txixi)+p_1^{b}(\txx,\txixi)\\
P_{(t,x)}(\txx,\txixi)&=(p_2^{\#}(\txx,\txixi)+p_1^{\#}(\txx,\txixi))+p_2^{b}(\txx,\txixi)+p_1^{b}(\txx,\txixi)+p_0(\txx,\txixi)\nonumber \\
&=q^{\#}(\txx,\txixi)+p_2^{b}(\txx,\txixi)+p_1^{b}(\txx,\txixi),
\end{align*} 
where 
\begin{equation}\label{sharp}
q^{\#}(\txx,\txixi)=(p_2^{\#}(\txx,\txixi)+p_1^{\#}(\txx,\txixi)+p_0(\txx,\txixi))\in S^{2}_{1,\delta},
\end{equation}
%
\begin{equation}
 p_2^{b}(\txx,\txixi)\in C^{\tau}S^{2-\tau\delta}_{1,\delta} \quad p_1^{b}(\txx,\txixi)\in C^{\tau-1}S^{1-(\tau-1)\delta}_{1,\delta}.
\end{equation}

\end{comment}
The symbol smoothing (Eq.(\ref{smoothing})) on $p_2,p_1$ gives a decomposition 
\begin{align*}
p_2(\uline{\txx},\uline{\txixi})&=p_2^{\#}(\uline{\txx},\uline{\txixi})+p_2^{b}(\uline{\txx},\uline{\txixi})\\
p_1(\uline{\txx},\uline{\txixi})&=p_1^{\#}(\uline{\txx},\uline{\txixi})+p_1^{b}(\uline{\txx},\uline{\txixi})\\
P_{(t,x)}(\uline{\txx},\uline{\txixi})&=(p_2^{\#}(\uline{\txx},\uline{\txixi})+p_1^{\#}(\uline{\txx},\uline{\txixi}))+p_2^{b}(\uline{\txx},\uline{\txixi})+p_1^{b}(\uline{\txx},\uline{\txixi})+p_0(\uline{\txx},\uline{\txixi})\nonumber \\
&=q^{\#}(\uline{\txx},\uline{\txixi})+p_2^{b}(\uline{\txx},\uline{\txixi})+p_1^{b}(\uline{\txx},\uline{\txixi}),
\end{align*} 
where 
\begin{equation}\label{sharp}
q^{\#}(\uline{\txx},\uline{\txixi})=(p_2^{\#}(\uline{\txx},\uline{\txixi})+p_1^{\#}(\uline{\txx},\uline{\txixi})+p_0(\uline{\txx},\uline{\txixi}))\in S^{2}_{1,\delta}(\R^8\times\R^8),
\end{equation}
\begin{equation}
 p_2^{b}(\uline{\txx},\uline{\txixi})\in C^{\tau}S^{2-\tau\delta}_{1,\delta}(\R^8\times\R^8) \quad p_1^{b}(\uline{\txx},\uline{\txixi})\in C^{\tau-1}S^{1-(\tau-1)\delta}_{1,\delta}(\R^8\times\R^8).
\end{equation}

Taking  $0\le \delta< 1$ so close to $1$ that $2-\tau\delta<2-\tau+\epsilon$ we have 
${v}\in H^{2+\sigma-\tau\delta}(\R^4\times \R^4)$ (notice this implies ${v}\in H^{1+\sigma-(\tau-1)\delta}(\R^4\times \R^4)$), and  we have 
\begin{equation}
q^{\#}(\uline{\txx},D_{\uline{\txx}})v=-(p_2^{b}(\uline{\txx},D_{\uline{\txx}})+p_1^{b}(\uline{\txx},D_{\uline{\txx}}))v=f,
\end{equation}
where $f\in H^{\sigma}(\R^4\times \R^4)$, since 
$p_2^b(\uline{\txx},D_{\uline{\txx}}) v\in  H^{\sigma}(\R^4\times \R^4)$ and 
$p_1^b(\uline{\txx},D_{\uline{\txx}}) v\in  H^{\sigma+1-\delta}(\R^4\times \R^4)$.

Now if  $(\uline{\tx_0},\uline{\txi_0},\uline{\ty_0},\uline{\teta_0})\notin \Char(P_{(t,x)})$  there are $C,c>0$ such that 
\begin{equation*}
|P_{(t,x)}(\uline{\txx},\uline{\txixi})|\ge c |\uline{\txixi}|^2 \text{ for } |\uline{\txixi}|\ge C
\end{equation*} 
in a conical neighbourhood $\Gamma$ that contains $(\uline{\tx_0},\uline{\txi_0},\uline{\ty_0},\uline{\teta_0})$.

Since  $p_2^{b}(\uline{\txx},\uline{\txixi})\in C^{\tau}S^{2-\tau\delta}_{1,\delta}$ and $p_1^{b}(\uline{\txx},\uline{\txixi})\in C^{\tau-1}S^{1-(\tau-1)\delta}_{1,\delta}$ there exists a $\tilde{C}>0$ such that 
\begin{align*}
|q^{\#}(\uline{\txx},\uline{\txixi})|&\ge C (1+|\uline{\txixi}|^2)-(1+|\uline{\txixi}|^2)^{\frac{2-\tau\delta}{2}}-(1+|\uline{\txixi}|^2)^{\frac{1-(\tau-1)\delta}{2}}\\
& \ge \tilde{C} (1+|\uline{\txixi}|^2) \text{ for large } |\uline{\txixi}|.
\end{align*}
Therefore  $(\uline{\tx_0},\uline{\txi_0},\uline{\ty_0},\uline{\teta_0})\notin \Char(q^{\#})$.

Since $q^{\#}\in S^{2}_{1,\delta}$ and  $(\uline{\tx_0},\uline{\txi_0},\uline{\ty_0},\uline{\teta_0})\notin \Char(q^{\#}) $ there is a microlocal parametrix with symbol  $\tilde{q}\in S^{-2}_{1,\delta}(\R^8\times\R^8)$ such that 
\begin{align*}
v+r(\uline{\txx},D_{\uline{\txx}})v&=\tilde{q}(\uline{\txx},D_{\uline{\txx}})q^{\#}(\uline{\txx},D_{\uline{\txx}})v
=\tilde{q}(\uline{\txx},D_{\uline{\txx}})f,
\end{align*}
where $(\uline{\tx_0},\uline{\txi_0},\uline{\ty_0},\uline{\teta_0})\notin WF(r(\uline{\txx},D_{\uline{\txx}})v)$ and 
$\tilde{q}(\uline{\txx},D_{\uline{\txx}})f\in H^{\sigma+2}(\R^4\times\R^4)$ which shows that  $(\uline{\tx_0},\uline{\txi_0},\uline{\ty_0},\uline{\teta_0})\notin WF^{\sigma+2}(\tilde q(\uline{\txx},D_{\uline{\txx}})f)$.
Since
\begin{align}
WF^{\sigma+2}(v)\subset WF^{\sigma+2}(\tilde q(\uline{\txx},D_{\uline{\txx}})f) \cup WF(r(\uline{\txx},D_{\uline{\txx}})v) ,
\end{align}
{we see that} $(\uline{\tx_0},\uline{\txi_0},\uline{\ty_0},\uline{\teta_0)}\notin WF^{\sigma+2}(v)$.

By definition of the wavefront set this means that $(\txx_0,\txixi_0)$ is not in the wavefront set of $v$ considered as a distribution on $M\times M$.
\end{proof}
\begin{rem}\label{local}
    In the proof presented above, we showed that the microlocal results are local estimates, which can be done within a chart in the cotangent bundle $T^*\R^8$. To streamline the discussion and avoid frequently alternating between the notation of the chart and the manifold, we will forego this distinction in Section 5. However, it is important to bear in mind that the proofs in that section are analogous to the one detailed above, involving localization within a chart.
\end{rem}

\begin{rem}\label{propagation1}
\begin{comment}
Applying the symbol smoothing directly to $P_{(t, x)}\in C^{\tau-1} S^2_{1,0}$ 
would leave us with $P_{(t, x)}^b\in C^{\tau-1}S^{2-{(\tau-1)}\delta}_{1,\delta}$.
The advantage of the decomposition in Theorem \ref{elliptic}  is that we obtain the remainder $p_2^b(\txx,D_{\txx})+p_1^b(\txx,D_{\txx})$ for $p^b_2\in C^\tau S^{2-\tau\delta}_{1,\delta}$ and  $p^b_1\in C^{\tau-1} S^{1-(\tau-1)\delta}_{1,\delta}$. Therefore,  given $u\in H_{comp}^{2+s-\tau\delta}(M\times M)$ we have $p^b_2(\txx,D_{\txx})u\in H^s(M\times M), p^b_1(\txx,D_{\txx})u\in H^{s+1-\delta}(M\times M)$ and hence  
$p_2^b(\txx,D_{\txx})u+p_1^b(\txx,D_{\txx})u\in H^s(M\times M)$ for $-{(1-\delta)}(\tau-1)<s<\tau-1$.
\end{comment}
Applying the symbol smoothing directly to $P_{(t,x)}\in C^{\tau-1}S^2_{1,0}$ would leave us with $P^b_{(t,x)}\in C^{\tau-1} S^{2-(\tau-1)\delta}_{1,\delta}$. The advantage of the decomposition in Theorem 3.3 with $p_1^b\in C^{\tau-1}S^{1-(\tau-1)\delta}_{1,\delta}$ and $p^b_2\in C^\tau S^{2-\tau\delta}_{1,\delta}$ is that the associated operators map a given $u\in H^{2+s-\tau\delta}$ to $H^s$ and $H^{s+1-\delta}$, respectively, for $-(1-\delta)(\tau-1)<s<\tau-1$, so that the sum is in $H^s$ instead of $H^{s-\delta}$.
\end{rem}
The theorem, below, will be crucial for our main result. Proofs can be  found in 
\cite[Proposition 6.1.D]{taylor} or \cite[Proposition 11.4]{tools}.
In \cite[p.215]{tools}, Taylor points out that Zygmund regularity $C_*^2$ for the metric suffices.
%\yafet{corrections?}
\begin{Theorem}\label{propagation}
Let $u\in \mathcal D'(M\times M)$ solve $P_{(t,x)}u=f$. Let $\gamma$ be an integral curve of the Hamiltonian vector field $H_{p_{2}}$ with $p_2$ as in Eq.\eqref{p0}.
{ If for some $s\in \R$,} we have $f\in H_{mcl}^s({\Gamma})$ and {$P_{(t,x)}^b u\in H_{mcl}^s({ \Gamma)}$}, where $\gamma\subset{\Gamma}$ with ${\Gamma}$ a conical neighbourhood and $u\in H_{mcl}^{s+1}(\gamma(0))$, then $u\in H_{mcl}^{s+1}({\gamma})$. 
\end{Theorem}

\begin{rem}\label{taylor}
If $u\in H_{comp}^{2+s-\tau\delta}$, then $P_{(t,x)}^b u\in H^s$, see Remark \eqref{propagation1}. 
Moreover, using the divergence structure of the operator one can show that, if $u\in H_{comp}^{1+s-\tau\delta}, f\in H^{s-1}, u\in H_{mcl}^{s}({\gamma}(0))$, then 
$u\in H_{mcl}^{s}({\gamma})$ for $-2(1-\delta)<s\le 2$; see {\rm\cite[p.210]{tools}} for details. 
\end{rem}

\begin{rem}
Notice that the $s\in\mathbb{R}$  is constrained by the microlocal regularity of $P^b_{(t,x)}u$ and not only that of $f$. In fact, one can use the stronger hypothesis that $u\in H^{s-\tau\delta}_{comp} (U)$ for a suitable domain  $U$,  regularity $\tau$ and $\delta\in(0,1)$ in order to guarantee that $P^b_{(t,x)}u\in H^s(U)\subset H^s_{mcl}(\Gamma)$ 
\end{rem}

\section{Support and Global Regularity of $K_G$ }

The following two lemmas contain the main results of this section. The  first lemma shows that only causally connected points belong to the support of $K_G$. The second lemma establishes that $K_G\in H_{loc}^{-1-\epsilon}(M\times M)$. 

{
\begin{Lemma}\label{causal}
Let $(\tx,\ty)\in M\times M$ be such that $\tilde x$ and $\tilde y$ are not causally related, i.e. $\tx\notin J(\ty)$. Then $(\tx,\ty)\notin\supp (K_G)$.
\end{Lemma}}
 \begin{proof}
Since the support of $K_G$ is the complement of the largest open set where $K_G$ vanishes, it is enough to show that there are open neighbourhoods $V$ of $\tilde x$ and 
$U$ of $\tilde y$ such that $K_G$ vanishes in $W=V\times U$.

We construct the sets $V$ and $U$ as follows: 
{For globally hyperbolic spacetimes there exists  a time function and a foliation by Cauchy surfaces i.e. $M=\mathbb{R}\times \Sigma$, see \cite[Theorem 1.1]{sanchez}, \cite[Theorem 5.9]{clemens}}. Let $\tx\in \{t\}\times\Sigma$ and $\ty\in \{s\}\times\Sigma$. Without loss of generality we assume $t\le s$.  
Since $M$ is globally hyperbolic, $J(\ty)\cap(\{t\}\times\Sigma)$ is compact and { by hypothesis does not contain $\tx$}.
Therefore there exists a neighbourhood $\tilde{V}$ of $\tx$ in $\{t\}\times\Sigma$ such that  $\overline{\tilde{V}}\cap (J(\ty)\cap(\{t\}\times\Sigma))=\emptyset$. By symmetry, $\ty \notin  J(\overline{\tilde V)}\cap(\{s\}\times \Sigma)=\overline{J(\tilde V)} \cap(\{s\}\times \Sigma)$,  and we thus also find a neighborhood $\tilde U$ of $\ty$ in $\{s\}\times \Sigma$ such that $\tilde U\cap J(\tilde V) \cap (\{s\}\times \Sigma) = \emptyset$. 

% Therefore there exists a neighbourhood $\tilde{V}\subset \{t\}\times\Sigma$ such that $\tx\in \tilde{V}$ and  $\tilde{V}\cap (J(\ty)\cap(\{t\}\times\Sigma))=\emptyset$.

%Similarly we can construct a neighbourhood $\tilde{U}$ such that  $\ty\in \tilde{U}$ and  $\tilde{U}\cap (J(\tilde{V})\cap(\{s\}\times\Sigma))=\emptyset$. { \color{blue}This neighbourhood can be constructed by taking a neighbourhood of $\ty$ contained in a ball with radius half the distance from the boundary of  $(J(\tilde{V})\cap(\{s\}\times\Sigma)$ to $\ty$. Notice that $\ty\not\in(J(\tilde{V})\cap(\{s\}\times\Sigma)$, otherwise there is a causal curve fron $\ty$ to ${\tilde{V}}$ contradicting  $\tilde{V}\cap (J(\ty)\cap(\{t\}\times\Sigma))=\emptyset$ }.
Now we consider the total domain of dependence of both sets i.e. $D(\tilde{U})$ and $D(\tilde{V})$\footnote{Given a subset $S$ of $M$, the domain of dependence of $S$ is the set of all points $p$ in $M$ such that every inextendible causal curve through $p$ intersects $S$.}. Notice that $J(D(\tilde{V}))\cap D(\tilde{U})=\emptyset$ and  $J(D(\tilde{U}))\cap D(\tilde{V})=\emptyset$. Otherwise, we could construct a causal curve between  $\tilde{U}$  and  $\tilde{V}$.
We define $V:=\text{Int}D(\tilde{V})$ and $U:=\text{Int}D(\tilde{U})$, see Figure \ref{support}.

{\begin{figure}[!ht]
\centering
\includegraphics[width=0.8\textwidth]{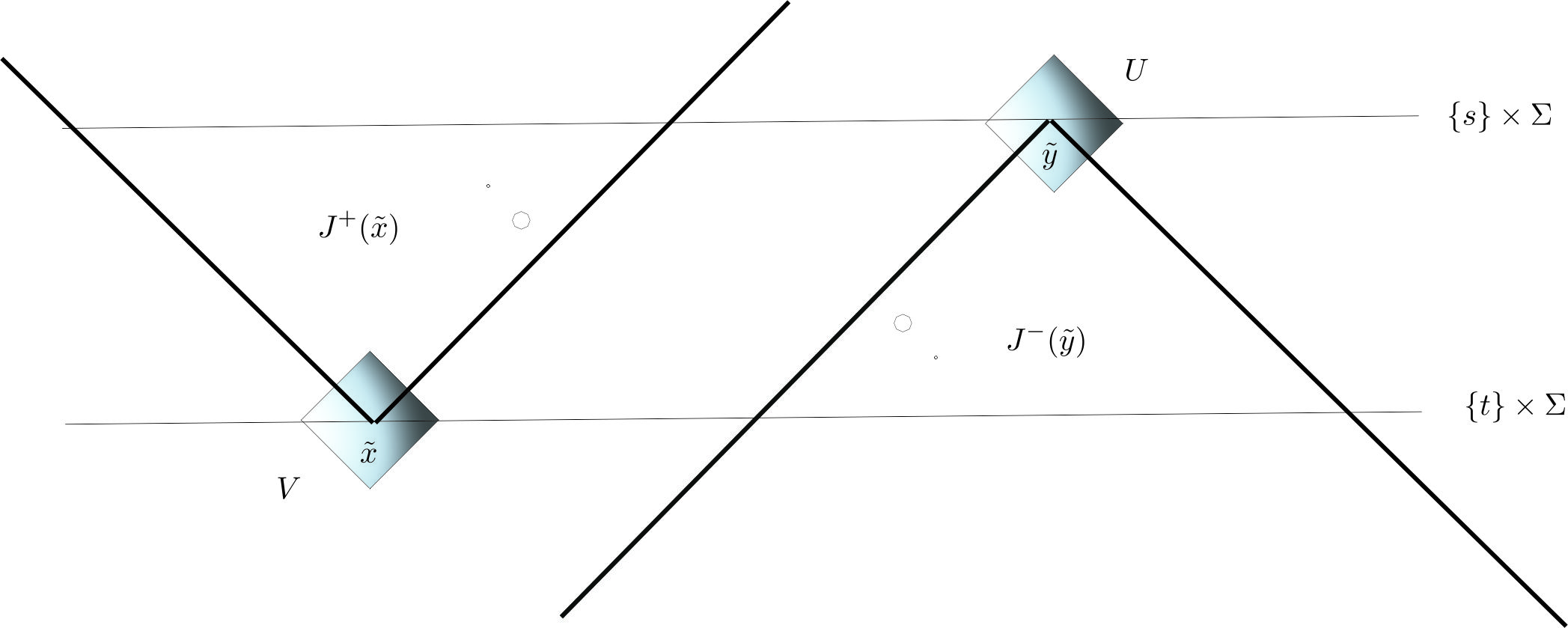}
\caption[]{ $U\cap J(V)=\emptyset$ and $V\cap J(U)=\emptyset$ }
\label{support}\end{figure}}
Now we show that $K_G$ vanishes in $W=V\times U$: 
Choose smooth functions  $\psi$ and $\phi$ with $\supp(\psi)\subset V$ and $\supp(\phi)\subset U$. 
Then 
\begin{align*}
K_G(\psi\otimes \phi)&=\ang{G(\psi),\phi}=\int_MG(\psi)\phi\sqrt{g}dx\\
&=\int_{J(\supp(\psi))\cap \supp(\phi)}G(\psi)\phi\sqrt{g}dx\\
&=\int_{J(V)\cap U}G(\psi)\phi\sqrt{g}dx=0.
\end{align*}
  
\end{proof}

\begin{rem}
    Notice that a totally analogous proof shows that $(\tx,\ty)\notin \supp(K_{G^\pm})$ if $\tx\notin  J^\pm(\ty)$.
\end{rem}

 Regarding the global regularity of the causal propagator for $C^{1,1}$ globally hyperbolic spacetimes, we find a slightly weaker result compared to the smooth case. Nevertheless, in the ultrastatic setting we show that the same regularity as in the smooth setting holds (Lemma \ref{normkg}). 

\begin{Lemma}\label{globalreg} Let $(M,g)$ be a $C^{1,1}$-globally hyperbolic spacetime.Then $K_G\in H_{loc}^{-1-\epsilon}(M\times M)$ for every $\epsilon>0$. 
\end{Lemma}

\begin{proof} \def\ul{\underline} 
We have to show that, given $\psi_1,\psi_2\in {\cal D}(M)$, the { Schwartz kernel of the } product $\psi_2G\psi_1$ is in 
$H^{-1-\varepsilon}(M)$ for every $\varepsilon>0$. Since the proof is local, we may assume (using possibly disconnected coordinate charts)  that $\psi_1$ and $\psi_2$ have their support in the same coordinate neighborhood for $M$. We will therefore work in $\mathbb R^4$, using the notation $\psi_1, \psi_2$ and $G$ also for the representations in local coordinates. In order to distinguish the standard variables and covariables  on $\mathbb R^4$ from those chosen for $M$ we shall denote them by $\ul x$, $\ul \xi$, etc.. Moreover, we choose $\psi_3,\psi_4\in {\cal D}(\mathbb R^4)$ supported in the same coordinate chart, satisfying $\psi_3\psi_2=\psi_2$ and $\psi_4\psi_3=\psi_3$.
Finally,  we denote by $\Lambda^s$, $s\in \mathbb R$, the pseudodifferential operator of order $s$ with symbol $(1+|\ul\xi|^2)^{s/2}$ on $\mathbb R^4$.   

We have
\begin{align}
 \psi_2 G\psi_1=\Lambda^{1+\varepsilon}\Lambda^{-1-\varepsilon}\psi_3\psi_2G\psi_1
=\Lambda^{1+\varepsilon} (\psi_4+(1-\psi_4))\Lambda^{-1-\varepsilon}\psi_3\psi_2 G\psi_1.
\end{align}
The operator $\psi_4\Lambda^{-1-\varepsilon}\psi_3\psi_2 G\psi_1$ maps $H^1(\mathbb R^4)$ to $H^{3+\varepsilon}_{comp}(\mathbb R^4)$ and therefore is a Hilbert-Schmidt operator. Hence it has an integral kernel in $L^2(\R^4\times \R^4)$. 
The operator $(1-\psi_4)\Lambda^{-1-\varepsilon}\psi_3$ is obviously smoothing, since $1-\psi_4$ and $\psi_3$ have disjoint support. Hence it maps $H^2(\R^4)$ to $H^\infty(\mathbb R^4)= \bigcap_sH^s(\R^4)$. 
But more is true:  In the identity 
\begin{eqnarray*}
\ul x_j(1-\psi_4)\Lambda^{-1-\varepsilon} \psi_3 = (1-\psi_4) \Lambda^{-1-\varepsilon} \ul x_j\psi_3
+ (1-\psi_4)[\ul x_j,\Lambda^{-1-\varepsilon}]\psi_3
\end{eqnarray*}
both operators on the right hand side map $H^2(\mathbb R^4)$ to $H^\infty(\mathbb R^4)$ (recall that $[\ul x_j,\Lambda^{-1-\varepsilon}]$ has the symbol $D_{\ul\xi_j}(1+|\ul \xi|^2)^{-1-\varepsilon}$). 
Iterating this identity we find that $(1+|\ul x|^{2N})(1-\psi_4)\Lambda^{-1-\varepsilon} \psi_3\in \mathcal B(H^2(\mathbb R^4),H^\infty(\mathbb R^4))$ for every $N\in \mathbb N$. 
Hence $(1-\psi_4)\Lambda^{-1-\varepsilon} \psi_3$ maps $H^2(\mathbb R^4)$ to  $\mathcal{S} (\mathbb R^4)$.%, the Schwartz space of rapidly decreasing functions. 

Therefore it also has an integral kernel in $L^2(\mathbb R^4\times \mathbb R^4)$. Denote for the moment the $L^2$-integral kernel of
$\Lambda^{-1-\varepsilon} \psi_3\psi_2G\psi_1$ by $k_A=k_A(\ul x,\ul y)$.  
Then the kernel $k=k(\ul x,\ul y)$ of $\psi_2G\psi_1$ is given by 
$$\Lambda^{1+\varepsilon}_{(\ul x)} k_A(\ul x,\ul y).$$
Here the notation $\Lambda^{1+\varepsilon}_{(x)} $ indicates that we view $\Lambda^{1+\varepsilon} $
as an operator on $\mathbb R^4\times \mathbb R^4$ that acts only with respect to the first copy of $\mathbb R^4$. In this sense, it is a pseudodifferential operator with symbol in the H\"ormander class $S^{1+\varepsilon}_{0,0}$ and thus maps $L^2(\mathbb R^4\times \mathbb R^4)$ to $H^{-1-\varepsilon}(\mathbb R^4\times \mathbb R^4)$. This shows the assertion. 

\end{proof}

\begin{rem}\label{regadvret}
    Notice that since only the mapping properties of $G$ were used we have also that $K_{G^+},K_{G^-}\in H^{-1-\epsilon}_{loc}(M\times M)$.
\end{rem}

\section{Proof of the Main Theorems}\label{stat}

A globally hyperbolic spacetime is given by a family of Riemannian metrics $\{h_t\}_{t\in\R}$ on $\Sigma$ and a function  $\beta(x,t)>0$ such that the spacetime metric  $(M,g)$, where $M=\mathbb{R}\times \Sigma$, is given by 

\begin{equation}
ds^2=\beta^2(t,x)dt^2-h_t,    
\end{equation}

see \cite[Theorem 1.1]{bernalsplitting}. We will assume that the regularity of the spacetime metric $g$ is $C^\tau$.

In this section we will prove the following results:

\begin{Theorem}\label{mains}
Let $(M,g)$ be a $C^{\tau}$ globally hyperbolic spacetime with $\tau>2$ and $K_G$ the causal propagator of the Klein-Gordon operator $P$. Then $$WF'^{-2+\tau-{\epsilon}}(K_G)\subset C$$ for
every ${\epsilon}>0$, $C$ as in Eq.\eqref{C}. % when is wavefrontset empty?
\end{Theorem}

\begin{Theorem}\label{mainmain}
For a $C^{\tau+2}$ globally hyperbolic spacetime with $\tau>2$ ,
    $$C\subset WF'^{-\frac{1}{2}}(K_G)\subset WF'^{\tau-\epsilon}(K_G)\subset C$$
holds for $0<\epsilon<\tau+\frac{1}{2}$.
\end{Theorem}

\begin{rem}\label{A}

    In the non-smooth case we cannot expect $G(f)\in C^\infty(M)$ even if $f\in {\cal{D}}(M)$ as a consequence of the fact that $G(f)$ solves the homogeneous Cauchy problem.  
We know from \cite[Proposition B.8]{adiabatic}, that for $f\in \cal{D}(M)$, 
$$ WF^s(G(f))\subset \{(\tx,\txi)\in T^*M;(\tx,\txi,\ty,0)\in WF^s(K_G) \text{ for some } y\in M\}.$$ Therefore $WF'^s(K_G)$ might contain points that are not in $C$.

\end{rem}

\begin{rem}\label{B}
    Since $K_G$ is antisymmetric, we have that for $\rho(\tx,\ty)=(\ty,\tx)$, $\rho^*K_G=-K_G$. This implies that if $(\tx,\txi,\ty,0)\in WF^s(K_G) \text{ for some } y\in M$, then   $(\ty,0,\tx,\txi)\in WF^s(K_G) \text{ for some } y\in M$.
    \end{rem}

\subsection{Proof of Theorem \ref{mains} }

Let $u\in H_{comp}^{1+s-\tau\delta}(M\times M)$ satisfy $P_{(t,x)}(\txx,D_\txx)u=0$.

Then also

$$\partial_\nu\left(\sqrt{|g|}g^{\mu\nu}\partial_\mu u\right)=0.$$

Using the decomposition $\sqrt{|g|}g^{\mu\nu}\partial_\mu=(\sqrt{|g|}g^{\mu\nu}\partial_\mu)^\#+(\sqrt{|g|}g^{\mu\nu}\partial_\mu)^b$  we obtain 

\begin{equation}
P_{(t,x)}(\txx,D_\txx)u=\frac{1}{\sqrt{|g|}}\partial_\nu\left((\sqrt{|g|}g^{\mu\nu}\partial_\mu)^\#u+(\sqrt{|g|}g^{\mu\nu}\partial_\mu)^bu\right).
\end{equation}
%where %{\color{blue}
%$\partial_{\nu}(\sqrt{|g|}g^{\mu\nu}\partial_\mu)^b:H^{s+1-\tau\delta}_{comp}(M\times M)\rightarrow H^{s-1}(M\times M)$ for $-(1-\delta)\tau<s<\tau$.

%In particular since $ K_G\in H_{loc}^{-1-\epsilon}(M\times M)$,  arguing locally as in Theorem \ref{elliptic},  is in the domain of $\partial_{\nu}(\sqrt{|g|}g^{\mu\nu}\partial_\mu)^b$ for $\tau>2$ choosing  $s=0$.

We state the behaviour outside the characteristic in this setting.

\begin{Lemma}\label{char3}
For $\tau>2$ and any $\tilde{\epsilon}>0$ 
\begin{equation}
WF^{-1-\tilde{\epsilon}+\tau}(K_G)\subset\Char(P_{(t, x)})\cap \Char(P_{(s, y)}) .
\end{equation} 
\end{Lemma}

\begin{proof}

As the statement is microlocal, we can work in local coordinates in $T^*(\R^4\times \R^4)$ and  consider $\varphi K_G$ for $\varphi\in {\cal{D}}(\R^4\times \R^4)$ with $\varphi=1$ near $\txx_0$.

Let $(\txx_0,\txixi_0)=(\tx_0,\txi_0,\ty_0,\teta_0)\not\in \Char(P_{(t,x)})$\footnote{Underscores to differentiate between the manifold points and points in $\R^8$ will be omitted. See Remark \ref{local}.}
Then, $0<\sqrt{|g(\tx)|}$ and $|g^{\mu\nu}(\tx)\sqrt{|g(\tx)|}\xi_\mu\xi_\nu|\ge C|\txixi|^2$ for sutable $C>0$ in a conic neighbourhood of $(\txx_0,\txixi_0)$.

In particular, $(\txx_0,\txixi_0)\not\in  \Char(\partial_{\nu}(\sqrt{|g|}g^{\mu\nu}\partial_\mu)^\#)$, so there exists a microlocal parametrix $\tilde{q}\in S^{-2}_{1,\delta}$ such that  

\begin{equation}\label{para}
\tilde{q}\partial_{\nu}(\sqrt{|g|}g^{\mu\nu}\partial_\mu)^\#=I+ r,
\end{equation}
where $r(\txx,D_{\txx})$ is microlocally smoothing near $(\txx_0,\txixi_0)$.

%Choose $\varphi_0,\varphi_1,\varphi_2\in {\cal{D}}(\R^4\times \R^4)$ with $\varphi_0$ supported near $\txx_{0}$, $\varphi_0(\txx_{0})\neq0$, 
%and $\varphi_1\varphi_0=\varphi_1$, $\varphi_2\varphi_1=\varphi_2$.

Since $P_{(t,x)}(\txx, D_\txx)K_G=0$, we have near $\txx_0$ 
\begin{align}\label{bras}
0&=\partial_\nu(\sqrt{|g|}g^{\mu\nu}\partial_\mu)K_G\\
&=\partial_\nu(\sqrt{|g|}g^{\mu\nu}\partial_\mu)^\#\varphi K_G+\partial_\nu(\sqrt{|g|}g^{\mu\nu}\partial_\mu)^b\varphi K_G,
\end{align}

Since $(\sqrt{|g|}g^{\mu\nu}\xi_\mu)^b\in C^\tau S^{1-\tau\delta}_{1,\delta}$ for every $0\le\delta<1$, we obtain a bounded map 
\begin{equation}
    \partial_\nu(\sqrt{|g|}g^{\mu\nu}\partial_\mu)^b:H^{s+1-\tau\delta}(\R^4\times \R^4)\rightarrow H^{s-1}(\R^4\times\R^4),
\end{equation}
$-(1-\delta)\tau<s<\tau \delta$.

Since $K_G\in H^{-1-\epsilon}_{loc}(M\times M)$ for every $\epsilon>0$ by Lemma \ref{globalreg}, we can choose $\delta$ such that $s=-2+\tau\delta-\epsilon>0$ so that by Eq. \eqref{bras}, we have locally
\begin{equation}
\partial_\nu(\sqrt{|g|}g^{\mu\nu}\partial_\mu)^\#\varphi K_G=-\partial_\nu(\sqrt{|g|}g^{\mu\nu}\partial_\mu)^b\varphi K_G\in H^{-3+\tau\delta-\epsilon}(\R^4\times \R^4).
\end{equation}

Applying the microlocal parametrix $\Tilde{q}$ we obtain

%\begin{equation}
    %\tilde{q}\varphi_1\partial_\nu(\sqrt{|g|}g^{\mu\nu}\partial_\mu)^\#\varphi K_G\in H^{-1+\tau\delta-\epsilon}(\R^4\times \R^4)
%\end{equation}

%Since $\varphi_2\Tilde{q}(1-\varphi_1)$ maps $H^{-1+\tau\delta-\epsilon}(M\times M)$ to rapidly decaying functions, see the proof of Lemma \ref{globalreg},

%We conclude that also 
\begin{equation}\label{square}
\tilde{q}\partial_\nu(\sqrt{|g|}g^{\mu\nu}\partial_\mu)^\#\varphi K_G\in H^{-1+\tau\delta-\epsilon}(\R^4\times \R^4).
\end{equation}

By Eq.\eqref{para}, Eq.\eqref{square} equals 

\begin{equation}
    (I+r(\txx,D_{\txx}))\varphi K_G.
\end{equation}

Hence, $K_G\in H^{-1+\tau\delta-\epsilon}(M\times M)$ microlocally near $(\txx_0,\txixi_0)$, so that $(\txx_0,\txixi_0)\not\in WF^{-1+\tau\delta-\epsilon}(K_G)$ for any $\epsilon>0, 0\le \delta<1$. Choosing $\delta$ appropriately we find that, for every $\tilde{\epsilon}>0$

\begin{equation}
    WF^{-1-\tilde{\epsilon}+\tau}(K_G)\subset \Char(P_{(t,x)}). 
\end{equation}

Arguing analogously for  $P_{(s,y)}$ we can see that
\begin{align}\label{preprechar}
WF^{-1+\tau-\tilde{\epsilon}}(K_G)\subset &\Char(P_{(t, x)})\cap \Char(P_{(s, y)}).
\end{align} 

\end{proof}

Notice that $$\Char(P_{(t, x)})\cap \Char(P_{(s, y)})=(\Char(P)\times \Char(P))\cup \mathcal{A}\cup \mathcal{B},$$ where $\mathcal{A}:=\{(\tx,0,\ty,\teta)\in T^*(M\times M):(\ty,\teta)\in \Char(P)\}$ and 
$\mathcal{B}:=\{(\tx,\txi,\ty,0)\in T^*(M\times M) :(\tx,\txi)\in \Char(P)\}$.

We will show now that the sets $\mathcal{A}$ and $\mathcal{B}$ do not belong to $WF^{-2+\tau-\Tilde{\epsilon}}(K_G)$. Nevertheless, for higher wavefront sets, that may not be the case, see Remark \ref{A} and Remark \ref{B}.

In order to show the result we will need the following lemma.

\begin{Lemma}\label{no1}
    $(\tx,\txi,\tx,\mu\txi)\notin WF^{-2+\tau-\tilde{\epsilon}}(K_{G^\pm})$ for $\mu\neq-1$.
\end{Lemma}

\begin{proof}
    Consider a point $(\ty,\teta)\neq(\tx,\txi)$ on the null bicharacteristic $\gamma(\tx,\txi)$, with $\ty\in J^{-}(\tx)$.
    Since $PG^+=I$, it holds
    \begin{equation}
        K_I=K_{PG^+}=P_{(t,x)}K_{G^+}
    \end{equation}
with wavefront set the conormal to the diagonal. As $\mu\neq-1$, $(\tx,\txi,\tx,\mu\txi)$ is not part of it, and neither are the points of the set $\gamma(\tx,\txi)\times\{(\tx,\mu\txi)\}$. Hence there exists an open conic neighbourhood $W$ of the set of all $(\Tilde{z},\tilde{\zeta},\tx,\mu\txi)\in T^*(M\times M),$ where $(\tilde{z},\Tilde{\zeta})$ lies on $\gamma(\tx,\txi)$ between $(\tx,\txi)$ and $(\ty,\teta)$, that does not intersect $WF(K_I)$. We can assume that the base point projection $\Pi(W)$ is relatively compact. We choose $\varphi\in \cal{D}(M\times M)$ with $\varphi=1$ on $\Pi W$. Then 
\begin{equation}
    \emptyset=WF(K_I)\cap W = WF(P_{(t,x)}K_{G^+})\cap W .
\end{equation}

Moreover, $P_{(t,x)}^{\#}(\varphi K_{G^+})=P_{(t,x)}(\varphi K_{G^+})-P_{(t,x)}^b(\varphi K_{G^+})$.

According to Remark 4.4, $K_{G^+}\in H^{-1-\epsilon}_{loc}(M\times M)$ for every $\epsilon>0$, therefore $P_{(t,x)}^b(\varphi K_{G^+})\in H^{-3-\epsilon+\tau}$.
We now apply Theorem \ref{propagation} with $u=\varphi K_{G^+}, s=-3-\Tilde{\epsilon}+\tau$, $\Gamma=W$, $f=P_{(t,x)}K_{G^+}\in H^\infty_{mcl}(W)$, $P_{(t,x)}^b(\varphi K_{G^+})\in H^{s}$. We have $\varphi K_{G^+}\in H^\infty_{mcl}$ near $(\ty,\teta,\tx,\mu\txi)$, since $(\ty,\tx)$ is not in the support of $K_{G^+}$. Hence, Theorem \ref{propagation} implies that $K_{G^+}\in H^{-2-\epsilon+\tau}_{mcl}$ also in a conic neighbourhood of $(\tx,\txi,\tx,\mu\txi)$, as this point lies on the integral curve of the Hamiltonian vector field for the principal symbol of $P_{(t,x)}$. Hence $(\tx,\txi,\tx,\mu\txi)\notin WF^{-2+\tau-\epsilon}(K_{G^+})$. In an analogous way we see that $(\tx,\txi,\tx,\mu\txi)\notin WF^{-2+\tau-\epsilon}(K_{G^-})$ by considering a point $(\ty,\teta)$ on $\gamma(\tx,\txi)$ with $\ty\in J^+(\tx)$.

\end{proof}
\begin{comment}
    
\begin{proof}
Let us assume $(\tx,\txi,\tx,\mu\txi)\in WF^{-2+\tau-\tilde{\epsilon}}(K_{G^+})$ with $\lambda\neq -1$. Now consider the null bicharacteristic   $\gamma(\tx,\txi)$ and notice that $P^\#K_{G^+}=P_{(t,x)}K_{G^+}-P_{(t,x)}^bK_{G^+}$.
Moreover,$P_{(t,x)}^bG^+\in H_{loc}^{-1-\Tilde{\epsilon}+\tau}(M\times M)$ by remark 4.4. Also, $P_{(t,x)}G^+=K_{I}$ and since the distributional kernel of the identity has as wavefront set the conormal to the diagonal, there exists  a conical neighbourhood $W$ containing the bicharacteristic $\gamma(\tx,\txi)\times\{(\tx,\mu\tx)\}$ such that $K_I\in H^s_{mcl}(W)$ for $s>0$. Therefore, we can apply  Theorem \ref{propagation} to conclude that $(\gamma(\tx,\tx),\tx,\mu\txi)\in WF^{-2+\tau-\tilde{\epsilon}}(K_{G^+})$. But since $\gamma(\tx,\txi)$ is a null bicharacteristic, it intersect in the first variable all Cauchy hypersurfaces, so we can find $\ty\in J^{-}(\tx)$ such that $\Pi(\gamma(\tx,\txi))=\ty$ where $\Pi$ is the projection of the first variable to $M$. However, this is a contradiction to the fact that $(\ty,\tx)\notin\supp(K_{G^+})$.
\end{proof}

\end{comment}

\begin{rem}
    Notice that the fact that the wavefront set of $K_I$ is the conormal to the diagonal does not allow one to repeat the same argument in the case $(\tx,\txi,\tx,-\txi)\in WF^s(K_{G^+})$.
\end{rem}

\begin{rem}
    A similar argument holds for the case $(\tx,\lambda\txi,\tx,\txi)\notin WF^{-2+\tau-\tilde{\epsilon}}(K_{G^\pm})$ by using $P_{(s,y)}$.
\end{rem}

\begin{Lemma}\label{char4}
For $\tau>2$ and any $\tilde{\epsilon}>0$ 
\begin{equation}
WF^{-2+\tau-\tilde{\epsilon}}(K_G)\subset\Char(P)\times \Char(P).
\end{equation} 
\end{Lemma}

\begin{proof}
Using Lemma \ref{char3} we just need to show that there are no points from the sets $\mathcal{A}$ or $\mathcal{B}$.     
Let $(\tx,\txi,\ty,0)\in {\cal{B}}\cap WF^{-2+\tau-\tilde{\epsilon}}(K_G)$ then by Theorem \ref{propagation}, we have that $(\gamma(\tx,\txi),\ty,0)\in WF^{-2+\tau-\tilde{\epsilon}}(K_G)$. Now $\ty=(s_1,y_1)$ for some $s_1\in \R,y_1\in \Sigma$. By global hyperbolicity $\gamma(\tx,\txi)$ intersects $\{s_1\}\times \Sigma$ in exactly one point with the covector $\chi\neq 0$. Since causally separated points are not in $\supp(K_G)$, the point of intersection  has to be $(s_1,y_1)$. 
Hence $(s_1,y_1,\chi,s_1,y_1,0)\in WF^{-2+\tau-\tilde{\epsilon}}(K_G)\subset(WF^{-2+\tau-\tilde{\epsilon}}(K_{G^+})\cup WF^{-2+\tau-\tilde{\epsilon}}(K_{G^-}))$. This is a contradiction to Lemma \ref{no1}.
A similar argument holds for points in $\mathcal{A}
$.\end{proof}

\begin{rem}
    The existence of symmetries allows one to show that the Sobolev wavefront set in Lemma 5.5 is already disjoint from the sets $\cal{A}$ and $\cal{B}$. %with the wavefront set order obtained using elliptic estimates.
    For example, if $M$ is stationary, $K_G$ is of the form $K_G(t-s,x,y)$. Therefore, one has the additional equation $(\partial_t+\partial_s)K_G=0$, that implies $WF^l(K_G)\subset \Char(\partial_t+\partial_s)$ for $l\in \R$. Moreover, $\Char(\partial_t+\partial_s)\cap \cal{A}=\emptyset$ and  $\Char(\partial_t+\partial_s)\cap \cal{B}=\emptyset$.
    A similar argument holds in the case of a sufficiently spatially symmetric spacetime, e.g. cosmological space of the form $ds^2=a(t)(-dt^2+dx^2+dy^2+dz^2)$. In this case,  $K_G$ is of the form $K_G(t,s,x_1-x_2,y_1-y_2,z_1-z_2)$ due to the spatial invariance.  
\end{rem}

Now we establish that points above the diagonal are of a specific form.

\begin{Lemma}\label{diag2}
 If $(\tx,\txi,\tx,\teta)\in WF^{-2+\tau-\tilde{\epsilon}}(K_G)$ for {$\tau>2$}, and some $\tilde{\epsilon}>0$, then $\teta=-\txi$.
\end{Lemma}

\begin{proof}
Suppose  $\tilde \eta$ and $\tilde \xi$ are linearly independent, i.e., $\teta\neq \mu\txi$ for $\mu\in \mathbb{R}$. By Lemma \eqref{char4} $(\tilde x,  \tilde \xi, \tilde x,\tilde \eta)\in \Char(P)\times\Char(P)$.
% i.e. $\teta$ is not colinear with $\txi$. Then we can c
Now we choose a Cauchy hypersurface $\Sigma_{ t_0}=\{t_0\}\times \Sigma$ such that the null geodesic with initial data $(\tx,\txi)$ and the null geodesic with initial data $(\tx,\teta)$ intersect it. These points of intersections are  unique by global hyperbolicity. Moreover, using the condition $\teta\neq\mu\txi$, we can choose $\Sigma_ {t_0}$ such that these points are distinct. We denote these points  by $(t_0,x_0), (t_0, y_0)$. Furthermore, they are not causally related. 
Now $K_G\in H_{loc}^{-1-\epsilon}(M\times m)$ so $\partial_{\nu}(\sqrt{|g|}g^{\mu\nu}\partial_\mu)^bK_G\in H^{-3-\epsilon+\tau\delta}(\R^4\times \R^4)$ and therefore if $(\tx,\txi,\ty,-\teta)\in WF^{-2+\tau-{\epsilon}}(K_G)$ then $(\gamma(\tx,\txi),\gamma (\ty,-\teta))\in { WF^{-2-\epsilon+\tau}}(K_G)$
 where $\gamma(\tx,\txi)$ is the null bicharacteristic with initial data $(\tx,\txi)$ and $\gamma(\tx,\teta)$ is the null bicharacteristic with initial data $(\tx,\teta)$.

In particular  
$(t_0,x_0, t_0, y_0)\in \Pi ({WF^{-\frac{1}{2}-\epsilon+\tau}}(K_G))$, 
where $\Pi$ is the projection from $T^*(M\times M)$ to $M\times M$. 
However,  this is a contradiction to Proposition \ref{causal}, since $(t_0,x_0, t_0, y_0)\notin\supp(K_G)$. Therefore, $\teta=\mu\txi$. 

Now as a consequence of the fact  that $K_G=K_{G^+}+K_{G^-}$ and $WF^s(K_G)\subset WF^s(K_{G^+})\cup WF^s(K_{G^-})$ for all $s$, Lemma \ref{no1} implies that $\mu=-1$.
\end{proof}

\emph{Proof of Theorem \ref{mains}}

Let $(\tx,\txi,\ty,-\teta)\in WF^{-2+\tau-{\epsilon}}(K_{G})$.
The propagation of singularities result (Theorem \ref{propagation})  implies that  $(\gamma(\tx,\txi),\gamma (\ty,-\teta))\in { WF^{-2-\epsilon+\tau}}(K_G)$, where  $\gamma(\tx,\txi)$ is the null bicharacteristic with initial data $(\tx,\txi)$ and $\gamma(\ty,-\teta)$ is the null bicharacteristic with initial data $(\ty,-\teta)$.

Now we choose a Cauchy surface $\Sigma_{t_1}=\{t_1\}\times \Sigma$ and suppose that $(t_1,x_1,\tilde{\xi}_1,t_1,x_2,\tilde\xi_2)\in (\gamma(\tx,\txi),\gamma(\ty,-\teta))\cap (\Sigma_{t_1}^2)$. 
By Lemmas \ref{causal} and  \ref{char4}, $(t_1, x_1,\txi_1),(t_1, x_2,\txi_2)\in \Char(P)$, $x_1=x_2$, and $\tilde \xi_2 = -\tilde \xi_1$.

Next we define a curve $\tilde{\gamma}:(-\infty,\infty)\rightarrow M$ as follows. First, we shift the parametrization  $\lambda$ in the definition of the null bicharacteristics so that 
$$\gamma(\tx,\txi)(t_1)=  (t_1,x_1,\txi_1), \quad   \gamma(\ty,-\teta)(t_1)=  (t_1,x_1,-\txi_1).$$  
Then, we denote by $\Pi:T^*M\to M$ the canonical projection and define two curves in $M$ by $$\gamma_1(\lambda):=\Pi(\gamma(\tx,\txi)(\lambda)),\quad \gamma_2(\lambda):=\Pi(\gamma(\ty,-\teta)(\lambda)).$$

Notice that we have ${\gamma_1}(t_1)=(t_1,x_1), \dot{\gamma_1}(t_1)=g^{-1}(\txi_1,\cdot)$ and ${\gamma_2}(t_1)=(t_1,x_1),\dot{\gamma_2}({t_1})=g^{-1}(-\txi_1,\cdot)$.
Moreover, we can assume that $\tx = \gamma_1(a)$ and $\ty =\gamma_2(b)$ for suitable $a, b\in \mathbb R$ with $a<t_1<b$. 

Finally, let 

\begin{equation}
\tilde{\gamma}(\lambda)=\begin{cases}
\gamma_1(\lambda)& \lambda\in(-\infty, t_1]\\
-\gamma_2(\lambda)& \lambda\in(t_1,\infty)
\end{cases}
\end{equation}
where $-\gamma_2$ denotes the curve with opposite orientation.

Then $\tilde \gamma(a) = \tx$, $\tilde \gamma(b) = \ty$; 
moreover $g(\cdot, {\dot{\tilde\gamma}})|_{T_{\tx}M}=\txi$, $g(\cdot, {\dot{\tilde\gamma}})|_{T_{\ty}M}=\teta$ and therefore, $\tilde{\gamma}$ is a null geodesic between $\tx$ and $\ty$  with cotangent vectors $\txi$ at $\tx$ and $\teta$ at $\ty$,
%
%i.e. $(\tx,\txi,\ty,\teta)\in C':=\{(\tx, \txi,\ty, -\tilde{\eta}) ;(\tx, \xi; \ty, \tilde{\eta}) \in C\}$.}
i.e. $(\tx,\txi,\ty,-\teta)\in C':=\{(\tx, \txi,\ty, -\tilde{\eta}) ;(\tx, \xi; \ty, \tilde{\eta}) \in C\}$, see Figure \ref{propfig}.

{\begin{figure}[!ht]
\centering
\includegraphics[width=0.8\textwidth]{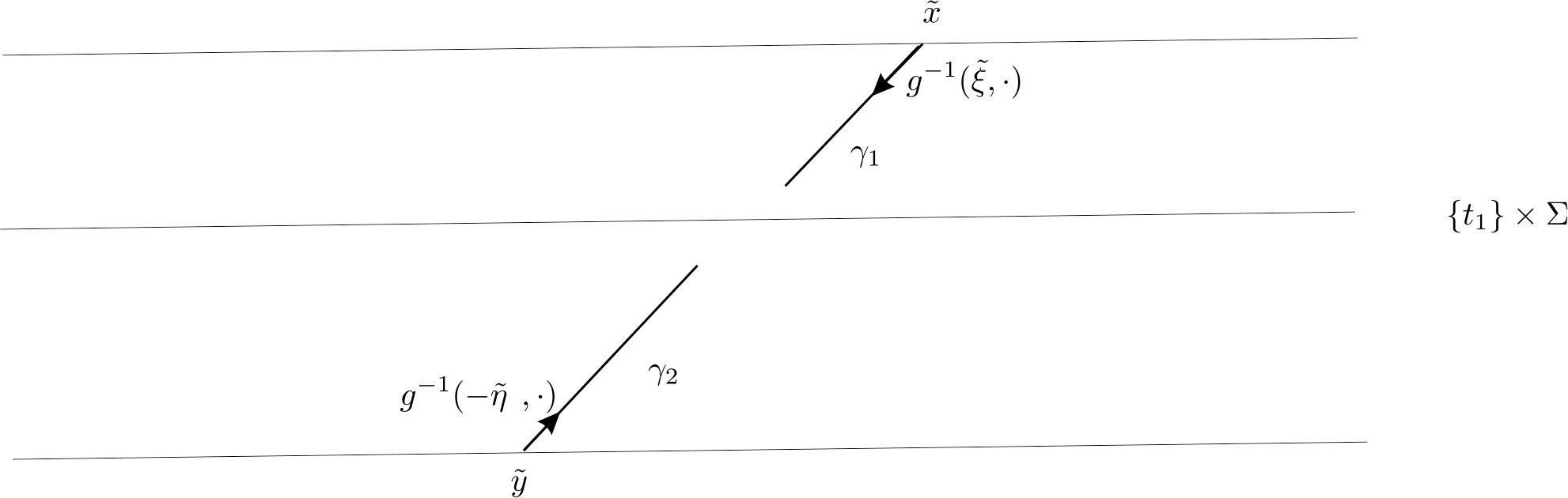}
\caption{$\gamma_1$ is a null geodesic that satisfies $\gamma(a)=\tilde{x},\dot{\gamma}_1(a)=g^{-1}(\xi,\dot)$ and $\gamma_2$ is a null geodesic that satisfies $\gamma(b)=\tilde{y},\dot{\gamma}_2(b)=g^{-1}(-\eta,\cdot)$ }
\label{propfig}\end{figure}}
This shows 
\begin{equation}
{WF^{-2-\epsilon+\tau}}(K_G)\subset C'
\end{equation}
or, equivalently ${WF'^{-2-\epsilon+\tau}}(K_G)\subset C$.

%Using  the definition of $WF^{l '}(u):= \{(\tx, \teta; \ty, -\tilde{\eta}) \in T^{*}(M\times M); (\tx, \txi; \ty, \tilde{\eta}) \in WF^l(u)\}$ and taking the  gives the result.

\subsection{Proof of Theorem \ref{mainmain}.}

%For the analysis of adiabatic states it is enough to work with the inclusion shown above. However, in the smooth case we have an equality of sets. 

%First we show the following lemma 

Now we show that $C$ is contained in $WF'^{-\frac{1}{2}}(K_G)$.

\begin{Lemma}\label{def}
    Let $P$ be the Klein-Gordon operator with $g\in C^{\tau+2},\tau>2$. Then   $C\subset WF'^{-\frac{1}{2}}(K_G)$
\end{Lemma}

\begin{proof}
{
 Using Proposition C.1 of \cite{fulling}, see also \cite{muller}, there exists an interpolating spacetime of regularity $C^\tau$, \((\bar{M}, \bar{g})\), which satisfies the following conditions: There exist times \(t_1\) and \(t_2\) such that for \(t < t_1\), \((\bar{M}, \bar{g})\) is isometric to a neighborhood of a Cauchy surface \(\tilde{\Sigma}\) of a smooth, globally hyperbolic spacetime $(M_s,g_s)$. Furthermore, for \(t > t_2\), \((\bar{M}, \bar{g})\) is isometric to a neighborhood of a Cauchy surface \(\Sigma\) of the non-smooth spacetime \(({M}, {g})\).

 Now if $K_{\Bar{G}}$ is the causal propagator associated to \((\bar{M}, \bar{g})\), its  restriction to $t<t_1$, denoted $K_{\Bar{G}}|_{t<t_1}$  corresponds to the smooth causal propagator \cite[Proposition 3.5.1]{bgp} and therefore 
 \begin{eqnarray*}
 WF'(K_{\Bar{G}}|_{t<t_1})
 =\Bar{C}\cap T^*( \{(t,x)\in\bar{M}; t<t_1 \}\times \{(t,x)\in\bar{M} ; t<t_1\}) ,   
 \end{eqnarray*} where $\Bar{C}$ denotes the canonical relationship associated to $\bar{g}$. 

Let $(\tx,\txi,\tx,-\txi)\in \bar{C}'$ in the non-smooth region i.e. $\tx=(t_3,x)$ with $t_3>t_2$.

    By global hyperbolicity the base point projections of the null bicharacteristics $\gamma(\tx,\txi)$ and $\gamma(\tx,-\txi)$ intersect the hypersurface $t=t_0<t_1$ at one unique point denoted $w$. Moreover, as a consequence of being in $\bar{C}'$, we have $(w,\chi,w,-\chi)=(\gamma(\tx,\txi)\times \gamma(\tx,-\txi))\cap(\tilde{\Sigma}_{t_0}\times \tilde{\Sigma}_{t_0})$.

Since we are in the smooth part, smooth theory implies, in particular, that $(w,\chi,w,-\chi)\in WF^{s}(K_{\Bar{G}}|_{t<t_1})$ for $-\frac{1}{2}\le s$ by combining \cite[Theorem 6.5.3]{duistermaat} and \cite[Propositon B.10]{adiabatic}. Now, an application of Theorem \ref{propagation} gives $(\tx,\txi,\tx,-\txi)\in WF^{s}(K_{\Bar{G}})$ for $-\frac{1}{2}\le s$.

Furthermore, by \cite[Theorem 5.10, Theorem 5.8]{green},  the restriction of $K_{\Bar{G}}$ to $t>t_2$, denoted $K_{\Bar{G}}|_{t>t_2}$,  in a neighbourhood of $\Sigma_{t_3}$ is the same as the restriction of the non-smooth causal propagator, $K_G$, associated to $({M}, {g})$.
Hence, $(\tx,\txi,\tx,-\txi)\in WF^{s}(K_{G})$. 

Another application of Theorem \ref{propagation} using the null bicharacteristics from $(M,g)$ gives $C'\subset WF^{-\frac{1}{2}}(K_G) $ i.e. $C\subset WF'^{-\frac{1}{2}}(K_G) $.
}

    %Using Proposition C.1 of \cite{fulling} we can deform a neighbourhood of a Cauchy surface $\tilde{\Sigma}$ of a smooth globally hyperbolic spacetime $(\Tilde{M},\tilde{g})$ into a neighbourhood of a Cauchy surface ${\Sigma}$ of the non-smooth spacetime $(M,g)$.
    %Since the proof requires change to normal coordinates, the interpolating spacetime metric $(\bar{M},\Bar{g})$ has regularity $C^{\tau}$.

    %Now, there is $t_1$ such that, for  $t<t_1$,  $K_G$ on $(\bar{M},\Bar{g})$ corresponds to the smooth causal propagator and therefore $WF'(K_G)=C$ in this region. 

    %Moreover, there is a $t_2$ such that for $t>t_2$ the spacetime is isometric to the non-smooth spacetime and $K_G$ corresponds to the non-smooth causal propagator. Let $(\tx,\txi,\ty,\teta)\in C'$ in the non-smooth region. 

    %By global hyperbolicity the base point projections of the null bicharacteristics $\gamma(\tx,\txi)$ and $\gamma(\ty,\teta)$ intersect the hypersurface $t=t_0<t_1$ at one unique point denoted $w$. Moreover, as a consequence of being in $C'$, we have $(w,\chi,w,-\chi)=(\gamma(\tx,\txi)\times \gamma(\ty,\teta))\cap(\tilde{\Sigma}_{t_0}\times \tilde{\Sigma}_{t_0})$.

%Since we are in the smooth part, smooth theory implies, in particular, that $(w,\chi,w,-\chi)\in WF^{s}(K_G)$ for $-\frac{1}{2}\le s\le-2+\tau-\Tilde{\epsilon}$. Now, an application of Theorem \ref{propagation} gives 
%$(\tx,\txi,\ty,\teta)\in WF^{s}(K_G)$.

\end{proof}

\emph{Proof of Theorem \ref{mainmain}.}
    
        The combination of Lemma \ref{def} and Theorem \ref{mains} gives the result.

\section{Appendix}
\subsection{Sobolev Spaces}\label{AppendixA}
$H^s(\R^n),\,s\in \R$, is the set of all tempered distributions $u$ on
$\R^n$ whose Fourier transforms ${\cal{F}}{u}$ are regular distributions
satisfying
\[ \| u\|^2_{H^s(\R^n)} := \int\!\ang{\xi}^{2s} |{\cal{F}}{u}(\xi)|^2\,d^n\xi <
\infty.\]
%\red{For a domain ${\cal U}\subset \R^n$ we let
%\[H^s({\cal U}) :=\{r_{\cal U}u;\;u\in H^s(\R^n)\}\]
%be the space of all restrictions to ${\cal U}$ of $H^s$-distributions
%on $\R^n$, equipped with the quotient topology
%\[ \|u\|_{H^s({\cal U})} :=\inf \{\|U\|_{H^s(\R^n)};\;U\in H^s(\R^n),
%r_{\cal U}U=u\}.\]}

Let $(M,g)$ be a (possibly) non-compact Riemannian manifold which
is geodesically complete. The Laplace-Beltrami operator $-\Delta_{g}$
is essentially selfadjoint if the regularity of the metric is $C^\tau$ for $\tau\ge2$, \cite[Theorem 2.4]{str}.
For lower regularity see Appendix \ref{AppendixB}.
By $H^s(M)$ we denote
the completion of ${\cal{D}}(M)$ with respect to the norm 
\[ \|u\|_{H^s(M)} := \|(I-\Delta_g)^{s/2}u\|_{L^2(M)}.\]

If $M$ is compact, $H^{s}(M)$ is independent of the metric. 

%If $M=(-t,t)\times \Sigma$ for convenience we consider only coordinates charts of the form $\kappa: (-t,t)\times U \rightarrow (t,-t)\times \mathbb{R}^4$

{ For an open subset $U$ of $M$ we define the local Sobolev spaces: 
\begin{eqnarray*}
H^s_{loc} (U) := \{u\in \Dp{M}; \varphi u\in H^s(M) \mbox{ for all }\varphi \in {\cal{D}}(U)\}.
\end{eqnarray*}{
and 
\begin{eqnarray*}
H^s_{comp} (U) := \{u\in \Dp{M}; u\in H^s(M) \text{ and }\supp(u)\subset U \text{ is compact}\}.
\end{eqnarray*}
}

{Notice that given a manifold $M$, the spaces $H^s_{loc}(U)$ and $H^s_{comp}(U)$ are independent of the Riemannian metric used to define the Sobolev spaces $H^s(M)$.}

For a compact $n$-dimensional manifold $\Sigma$ we can also define Sobolev spaces on $\mathbb R\times \Sigma$  relying on local coordinates. Namely, suppose $\{U_j:j\in J\}$ is an open cover of $\Sigma$ by coordinate charts and $\{\varphi_j:j\in J\}$ is a subordinate partition of unity. Given a function  $u$ on $\mathbb R\times \Sigma$, we say that $u\in \tilde H^s(\mathbb R\times \Sigma)$ provided that, using local coordinates on $\Sigma$,  $\varphi_j(x)  u(t,x) \in H^s(\mathbb R\times \mathbb R^n)$  for  $j=1,\ldots, J$ (more formally: For the coordinate map $\kappa_j:U_j\to \R^n$, we have 
$(id\times \kappa_{j*})(\varphi_j u)\in H^s(\R\times \Sigma)$).
For integer $k$, this  is equivalent to asking that, for all multi-indices $\alpha$ with $|\alpha |\le k$, we have $\partial^\alpha_{t,x} u\in L^2(\R\times \R^n)$ in local coordinates. 
Moreover, $\R\times \Sigma$ is a manifold of bounded geometry and the Sobolev spaces introduced in this setting coincide with the spaces $\tilde H^s$, see e.g. Theorem 3.9 in \cite{GS13}.  

%If $\Sigma$ has a $C^k$-structure, $1\le k\in \N$, then the spaces $\tilde H^s$ are independent of the choice of the coordinates in view of the coordinate invariance of the standard Sobolev spaces.  

\begin{Lemma}\label{Sobolev} 
Let $g= dt^2 + h_{ij}dx^i dx^j$ be an ultrastatic metric of regularity $C^\tau$  on $\R\times \Sigma$ with $\tau >1$. Then 
$$ H^s(\R\times \Sigma)  = \tilde H^s(\R\times \Sigma), \quad 0\le s\le 2,$$
i.e. the two Hilbert spaces coincide up to equivalent norms.
\end{Lemma} 

\begin{proof}
The assertion is obvious for $s=0$, when $\tilde H^0(\R\times \Sigma) = L^2(\R\times \Sigma) = H^0(\R\times \Sigma)$. We have 
$$H^s(\R\times \Sigma) = \mathscr D((I-\Delta_g)^{s/2}) = 
[L^2(\R\times \Sigma),\mathscr D(I-\Delta_g)]_{s/2},$$
where the first equality holds by definition and the second is \cite[Section I.2.9]{amann} for complex interpolation.

In view of the interpolation property for the standard Sobolev spaces, it is sufficient to show the assertion for $s=2$. 
Assuming that $\tau>1$, the operator $I-\Delta_g$ is  strongly elliptic with coefficients in $C^{\tau-1}$. 
By elliptic regularity, its maximal domain is $\tilde H^2(\R\times \Sigma)$.
This is a well-known fact, although a reference seems to be hard to find. 
In order to see it we first note that, by Lax-Milgram's theorem, every $u\in L^2(\R\times \Sigma) $ with 
$\Delta_g u \in L^2(\R\times \Sigma)$ belongs to $H^1(\R\times \Sigma)$. 
Symbol smoothing as in Remark \ref{propagation1} then shows that $u$ even belongs to $\tilde H^2(\R\times \Sigma)$. Hence the maximal domain is a subset of $\tilde H^2(\R\times \Sigma)$.

%It can be deduced, however, by a suitable modification of  \cite[Theorem 9.15]{gil} or \cite[Section 5]{gro} or else by a freezing-of-coefficients technique as in \cite[Section 5.3]{DHP}. 
%A proof follows from \cite[Theorem  1.8]{KS22} %\cite[Theorem 9.5]{gil} for }
%$X=\R\times \Sigma$, $\partial X=\emptyset$, $A=I-\Delta_g$ (note that the domain is independent of the choice of the parameter $\lambda$ used there). 
The minimal domain is also $\tilde H^2(\mathbb R\times \Sigma)$, since $\mathcal D(\mathbb R\times \Sigma)$ is dense in $\tilde H^2(\mathbb R\times \Sigma)$. Hence $\mathscr D(I-\Delta_g)=\tilde H^2(\mathbb R\times \Sigma)$.   

%But then, given a function $u$ on $\R\times \Sigma$, we have  $u\in H^{2k} (\R\times \Sigma) $ if and only if $(I-\Delta_g)^ku\in L^2(\R\times \Sigma)$. This, in turn. is equivalent to the fact that $\partial^\alpha_{t,x} \varphi_ju \in L^2(\R\times \Sigma)$, $|\alpha|\le 2k$, which, as we saw above, also characterizes the functions in $\tilde H^{2k}(\R\times \Sigma)$. 
\end{proof} 

\begin{rem}An analogous construction can be performed for $\R^2\times \Sigma^2$ and the analog of Lemma \ref{Sobolev} holds. 
\end{rem} 
}

{
\subsection{Essential Self-adjointness of the Laplace-Beltrami Operator}\label{AppendixB}

\begin{Theorem}\label{self}
Let $(\Sigma,h)$ be a smooth compact $n$-dimensional manifold equipped with a Riemannian metric of regularity $C^{1}(\Sigma)$. Then the Laplace-Beltrami operator $\Delta_{h}$ is essentially self-adjoint.
\end{Theorem}

 We follow Strichartz's article \cite{str} that uses the following criterion \cite[Theorem X.1]{reed}.
 
 \begin{Theorem}\label{stri}
 Let $A$ be any closed negative-definite symmetric, densely defined operator on a Hilbert space $H$. Then $A=A^*$ if and only if there are no
 eigenvectors with positive eigenvalue in the domain of $A^*$.
 \end{Theorem}

 Now we will state the following helpful  result %by Yau \cite {str, yau}
 
 \begin{Proposition}\label{eigen}
  Let $u$ be an $L^2(\Sigma)$ function that satisfies $\Delta u=\lambda u$ for some $\lambda>0$. Then $u$ is identically zero.  
 \end{Proposition}

 {\emph{Proof.}} Let $u$ be a weak solution which by elliptic regularity satisfies $u\in H^2({\Sigma})$. 
 %\red{This should also follow in this way: We have $\Delta u\in L^2$ and $\Delta^\flat u \in L^2$ for $\tau>2$. Hence $\Delta^\#u=\Delta u -\Delta^\flat u\in L^2$. Applying a parametrix to $\Delta^\#$ to this we get $u\in H^2$.}  
% Moreover, this regularity of $u$ combined with the regularity of $h$ is enough to apply the divergence theorem \cite[Proposition A.1]{stokes}. 

Hence, 
 \begin{align}
  \lambda\left(u, u\right)_{L^{2}(\Sigma)} =&(\Delta u ,u)_{L^{2}(\Sigma)}= -\left(du, du\right)_{L^{2}(\Sigma)}
 \end{align}

 Now $\lambda> 0$ so we have $u=0$.
\qed

\emph{Proof of Theorem \ref{self}.}
By direct computation $\Delta_{h}$ is negative-definite and symmetric. That it is densely defined follows from the density of ${\cal{D}}(M)$ in $L^2(M)$ for continuous metrics (see \cite[Proposition 7]{lashi} for even rougher cases).  The application of Theorem \ref{stri} taking into account Proposition \ref{eigen} gives the result.
\qed

For the non-compact case, one could follow the construction in Strichartz's article. However, suitable modifications are required under the regularity of Theorem \ref{self}. For example, one would have to use an integral distance as in \cite[Theorem 5.11]{integral}  to show the desired properties of the approximations to unity. Then, one would need to verify that the elliptic regularity results hold in that situation as well. %An alternative approach can be found in \cite[Lemma A.1]{der2}.

Since we are only interested in the case of $M=\R^2\times \Sigma ^2$ and the operator $2mI-\partial_{tt}-\partial_{ss}-\Delta_{h_{x}}-\Delta_{h_{y}}$ under $C^{1,1}$ regularity assumptions,  we will proceed in a different manner. 

\begin{Lemma}\label{selfi}
The operator  $2mI-\partial_{tt}-\partial_{ss}-\Delta_{h_{x}}-\Delta_{h_{y}}$, where $h_{x},h_{y}$ are Riemannian metrics of regularity $C^{1,1}$, is essentially self-adjoint with domain $H^2(\R^2\times \Sigma ^2)$.  
\end{Lemma}
\begin{proof}

 %\cite[Theorem VIII.33]{reed1}
By \cite[Lemma 2.1]{reed3} we obtain that  $-\partial_{tt}-\partial_{ss}-\Delta_{h_{x}}-\Delta_{h_{y}}$ is essentially self-adjoint in 
$L^2(\R^2\times\Sigma^2)$ with domain 
${\cal{D}}(\R)\otimes {\cal{D}}(\R)\otimes C^\infty(\Sigma^2)$. Since ${\cal{D}}(\R)\otimes {\cal{D}}(\R)\otimes C^\infty(\Sigma^2)$ is dense in $H^2(\R^2\times \Sigma ^2)$ which carries the graph norm of  $-\partial_{tt}-\partial_{ss}-\Delta_{h_{x}}-\Delta_{h_{y}}$, we obtain that the closure of the domain is $H^2(\R^2\times \Sigma^2)$.
Now  $-\partial_{tt}-\partial_{ss}-\Delta_{h_{x}}-\Delta_{h_{y}}$  and $2mI$ commute and are self-adjoint.  By \cite[Lemma 4.16.1]{putman} $2mI-\partial_{tt}-\partial_{ss}-\Delta_{h_{x}}-\Delta_{h_{y}}$  is self-adjoint with domain $H^2(\R^2\times \Sigma ^2)$.
\end{proof}
}

\subsection{An Equivalent Sobolev Norm}

The main results of this section are the following proposition and Corollary \ref{normneg}.
{
\begin{Proposition}\label{normeigen}
Let {$\Sigma$ be a compact manifold and} $\{\phi_j\otimes\phi_k; j,k=1,2,\ldots\}$ be an orthonormal basis of $L^2(\Sigma)\otimes_H L^2(\Sigma) $ associated to the eigenfunctions $\{\phi_j\}$ of the operator $mI-\Delta_h$, $m>0$. 
Writing   $u\in L^2((\R\times \Sigma)\times (\R\times \Sigma)) \cong
 L^2(\mathbb{R}^2\times \Sigma^2)\cong  L^2(\mathbb{R}^2)\otimes_{H} L^2(\Sigma)\otimes_{H} L^2(\Sigma)$ in the form 
\begin{eqnarray}\label{repu}%\lefteqn{}\\
{u(t,s, x,y) =\sum_{j,k}u_{jk}(t,s)\phi_j(x)\phi_k(y)}  \quad \text{with }
u_{jk}=\ang{u,\phi_j\otimes\phi_k}\in  L^2(\mathbb{R}^2),
\end{eqnarray}
we obtain the following alternative description of the Sobolev spaces:  For $0\le s\le2$ 
\begin{align*}
&H^{s}(\R^2\times \Sigma^2)\\ %\mathbb{R}^2\times \Sigma^2)\\
&=\{u\in \mathcal S'(\mathbb{R}^2\times \Sigma^2);
\sum_{j,k}\int_{\mathbb{R}^2}(\xi_0^2+\eta_0^2+\lambda_j^2+\lambda_k^2)^{s}|(\mathcal F{u}_{jk)}(\xi_0, \eta_0)|^2d\xi_0d\eta_0<\infty\}.\nonumber
\end{align*}
\end{Proposition}
}

Here $\mathcal S'(\mathbb{R}^2\times \Sigma^2)$ is the dual space to $\mathcal S(\mathbb{R}^2\times \Sigma^2) := \mathcal S(\R^2) \hat\otimes_\pi C^\infty(\Sigma^2)$. 
First we show the result in the particular case  $s=2$:

\begin{Lemma}\label{lemmah2}
\begin{align*}
&H^{2}(\R^2\times \Sigma^2) \\%\mathbb{R}^2\times \Sigma^2)\\
&=\{u\in \mathcal S'(\mathbb{R}^2\times \Sigma^2); \sum_{j,k}\int_{\mathbb{R}^2}(|\xi_0|^2+|\eta_0|^2+\lambda_j^2+\lambda_k^2)^{2}|(\mathcal F{u}_{jk})(\xi_0, \eta_0)|^2d\xi_0d\eta_0<\infty\}.
\end{align*}
\end{Lemma}

\begin{proof}
By definition (see Appendix \ref{AppendixA})
\begin{equation}\label{Hs}
H^{2}(\R^2\times \Sigma^2)%\mathbb{R}^2\times \Sigma^2)
=\{u\in L^2(\mathbb{R}^2\times \Sigma^2); (I-\partial_{tt}-\partial_{ss}-\Delta_{h{_{x}}} -\Delta_{h{_{y}}})u\in L^2(\mathbb{R}^2\times \Sigma^2)\},
\end{equation}
where we have equipped $\Sigma\times \Sigma$ with the product metric $\tilde{h}$ induced by the metric $h$ on each of the  components
% of $(\mathbb R\times \Sigma)\times (\mathbb R\times \Sigma)$  
so that   
$\Delta_{\tilde{h}}=\partial_{tt}+\partial_{ss}+\Delta_{h_{x}} +\Delta_{h_{y}}$, where  $\Delta_{h_x}$ and $\Delta_{h_y}$ are the Laplacians for the metric $h$ on the first and second component of $\Sigma\times \Sigma$, respectively. 
% $h{_{x}}$ a non-smooth Riemannian metric acting on the first component of $M\times M$ and $h{_{y}}$ is a non-smooth Riemannian metric acting on the second component of $M\times M$. Hence,  u
Since $m>0$ we may (at the expense of obtaining an equivalent norm) replace $I$ in  
Eq.\eqref{Hs} by $2mI$. 
Writing $u\in H^2(\mathbb R^2\times \Sigma^2)\subset L^2(\mathbb R^2\times \Sigma^2)$ in the form \eqref{repu}  and using the orthonormality of the set $\{\phi_l\}_l$ in $L^2{(\Sigma)}$ we obtain 

\begin{align*}
&\|u\|^2_{H^2(\mathbb{R}^2\times \Sigma^2)}\\
&:=\int_{\mathbb{R}^2}\int_{\Sigma\times \Sigma}|(2mI-\partial_{tt}-\partial_{ss}-\Delta_{h{_{x}}} -\Delta_{h{_{y}}})u|^2\sqrt{h(x)}\sqrt{h(y)}dtdsdxdy\\
&=\int_{\mathbb{R}^2}\sum_{j,k}|-(\partial_{tt}+\partial_{ss})u_{jk}(t,s)+\lambda^2_{j}u_{jk}(t,s)+\lambda_{k}^2u_{jk}(t,s)|^2dtds.
\end{align*}

Applying the Fourier transform in $(s,t)$, Plancherel's theorem shows that 
\begin{align*}
&\|u\|^2_{H^2(\mathbb{R}^2\times \Sigma^2)}=\sum_{j,k}\int_{\mathbb{R}^2}
\left(\xi_0^2+\eta_0^2+\lambda^2_{j}+\lambda_{k}^2\right)^2
|({\cal{F}}u_{jk})(\xi_0,\eta_0)|^2d\xi_0 d\eta_0
\end{align*}

which proves the result. 
\end{proof}

Before proving the main proposition we state the following  result found in  Amann \cite[I.(2.9.8)]{amann}.

\begin{Theorem}\label{domain}
Let $A$ be a non-negative self-adjoint operator.  Then we have the following relation for the domains of the powers of $A$: 
$$\mathscr D(A^{(1-\theta)\alpha+\theta\beta})
=[\mathscr D(A^\alpha),\mathscr D(A^\beta)]_\theta$$
for $0\le {\rm Re}\,\alpha< {\rm Re}\,\beta$ and $0<\theta<1$. Here $[\cdot,\cdot]_\theta$ denotes complex interpolation.
\end{Theorem}
%pg33

%\begin{proof}\emph{Proposition \ref{normeigen}}

{\em Proof of Proposition} \ref{normeigen}. 
Since $(\R\times \Sigma)\times (\R\times\Sigma)$ 
%$\mathbb{R}^2\times (\Sigma\times\Sigma)$ 
is a complete manifold, the operator $2mI-\Delta_{\tilde{h}}$ is  positive and self-adjoint (see Appendix \ref{AppendixB}). Using Theorem \ref{domain} we obtain  for $0<\theta<1$ 

\begin{eqnarray*}%\label{}
\lefteqn{\nonumber
H^{2\theta}(\R^2\times \Sigma^2)
=\mathscr D((2mI-\Delta_{\tilde{h}})^\theta)}\\
&=\{u\in \mathcal S'(\mathbb{R}^2\times \Sigma^2); (2mI-\Delta_{\tilde{h}})^{\theta}u\in L^2(\mathbb{R}^2\times \Sigma^2)\}.
\end{eqnarray*}
Since $2mI-\Delta_{\tilde{h}}
= 2mI-\partial_{tt}-\partial_{ss}-\Delta_{h_{x}} -\Delta_{h_{y}}$
can be written as a multiplication operator in the form 
\begin{equation*}
 (2mI-\Delta_{\tilde{h}})u=
 \sum_{j,k}\phi_j(x)\phi_k(y)\int_{\mathbb{R}^2}e^{i(\xi_0 t+\eta_0 s)}\left(\xi_0^2+\eta_0^2+\lambda^2_{j}+\lambda_{k}^2\right)(\mathcal Fu_{jk})(\xi_0,\eta_0) d\xi_0 d\eta_0,
\end{equation*}
we infer from the orthonormality of the $\phi_j$  that $(2mI-\Delta_{\tilde{h}})^\theta u$
in $L^2(\mathbb{R}^2\times \Sigma^2)$, if and only if
%the domain of the power $(2mI-\Delta_{\tilde{h}})^r$ are the functions $u\in L^2(\mathbb{R}^2\times (\Sigma\times\Sigma))$ such that 
%
\begin{eqnarray*}
\lefteqn{%
\Big\|\sum_{j,k}\phi_j(x)\phi_k(y)\int_{\mathbb{R}^2}e^{i(\xi_0 t+\eta_0 s)}\left(\xi_0^2+\eta_0^2+\lambda^2_{j}+\lambda_{k}^2\right)^\theta({\cal{F}}u_{jk})(\xi_0,\eta_0) d\xi_0 d\eta_0\Big\|^2_{L^2(\mathbb R^2\times \Sigma^2)}}\\
&=&\sum_{j,k}\int_{\mathbb{R}^2}\left(\xi_0^2+\eta_0^2+\lambda^2_{j}+\lambda_{k}^2\right)^{2\theta}|({\mathcal{F}}u_{jk})(\xi_0,\eta_0)|^2 d\xi_0 d\eta_0
<\infty.
\end{eqnarray*}
This establishes the required equivalence.
\hfill $\Box$
%\end{proof}
{
\begin{Corollary}\label{normneg}
For $-1\le \theta\le 0$ we obtain by $L^2$-duality that
\begin{align*}
&H^{2\theta}(\mathbb{R}^2\times \Sigma^2)
= (H^{-2\theta}(\mathbb R^2\times \Sigma^2))'\\
&=\{u\in \mathcal S'(\mathbb{R}^2\times \Sigma^2); \sum_{j,k}\int_{\mathbb{R}^2}(|\xi_0|^2+|\eta_0|^2+\lambda_j^2+\lambda_k^2)^{2\theta}|\mathcal Fu_{jk}(\xi_0, \eta_0)|^2d\xi_0d\eta_0<\infty\},
\end{align*}
with $u_{j,k} = \ang{u,\phi_j\otimes \phi_k}\in \mathcal S'(\R^2)$. 
\end{Corollary}

}

\subsection{The Ultrastatic Case}

 In this case we consider a Lorentzian metric $g$ on  $M=\R\times \Sigma$ {with $\Sigma$ compact} of the form  
$$ds^2=dt^2-h_{ij}(x)dx^idx^j$$ 
where $h_{ij}(x)$ are the components of a time independent   Riemannian metric
of H\"older regularity $C^{\tau}$ (when $\tau \in \N$ we will  consider the Zygmund spaces $C_*^\tau$, introduced in Definition \ref{Holder}).

The Klein-Gordon operator $P$ on $M$ is  
\begin{equation}\label{kg}
P\phi=\partial_{tt}\phi-\Delta_h\phi +m^2\phi
\end{equation}
with $\Delta_h\phi=\frac{1}{\sqrt{h}}\partial_{x^i}(h^{ij}\sqrt{h}\partial_{x^j}\phi)$ and $m>0$. 

The causal propagator $G$ is given by $\displaystyle{-\frac{\sin(A^\frac{1}{2}(t-s))}{A^{\frac{1}{2}}}}$ where $A:=-\Delta_h+m^2$ is self-adjoint on $L^2 (\Sigma)$, see Appendix \ref{AppendixB}. 

Moreover, the spectrum of $A$ is a discrete set of positive eigenvalues which we denote by  $\{\lambda_j^2; j=1,2,\ldots\}$, listed  according to their (finite) multiplicity. The associated  
set $\{\phi_j\}_{j\in\mathbb{N}}$ of normalized real eigenfunctions is an orthonormal basis of $L^2(\Sigma)$, see \cite[Theorem 5.8]{spin}.  % The eigenvectors $\phi_j$ satisfy$A\phi_j=\lambda_j^2\phi_j$.
For $u,v\in {\mathcal D}(M)$  we have $G(v)\in \mathcal D'(M)$ given by 
\begin{align}
&\langle G(v),u\rangle:=\\
&=\int_{\Sigma}\int_{-\infty}^{\infty}\left(-\frac{\sin(A^\frac{1}{2}(t-s))}{A^{\frac{1}{2}}}v\right)(t,x)u(t,x)\sqrt{h(x)}dxdt\nonumber\\
&=-\int_{M}\left(\int_{-\infty}^{\infty}\sum_j \lambda_j^{-1} \sin(\lambda_j(t-s))\phi_j(x)\int_\Sigma\phi_j(y)v(s,y)\sqrt{h(y)}dyds\right)u(t,x)\sqrt{h(x)}dxdt. \nonumber
\end{align}
%
%and the kernel of the causal propagator acting  on the densities $u\sqrt{h(x)}\otimes v\sqrt{h(y)}$ is given by
%\begin{align}
%&K_G(u\otimes v)=&\\
%&=-\int_{M\times M}\sum_j \lambda_j^{-1} \sin(\lambda_j(t-s))\phi_j(x)\phi_j(y)\sqrt{h(y)}\sqrt{h(x)}  u(t,x) v(s,y)dsdydtdx,\nonumber
%\end{align}
{
Using that $\langle G(v),u\rangle=\langle K_G,v\otimes u\rangle$ gives the singular integral kernel representation
\begin{equation}\label{kernel}
K_G(t,x;s,y)=-\sum_j \lambda_j^{-1} \sin(\lambda_j(t-s))\phi_j(x)\phi_j(y).
\end{equation}
}
\subsubsection{Global Regularity}
Now we show in Lemma \ref{normkg} that in ultrastatic spacetimes the global regularity of the causal propagator is the same as in the smooth case

\begin{Lemma}\label{normkg}
$K_G\in H_{loc}^{-\frac{1}{2}-\epsilon}(M\times M)$ for every $\epsilon>0$. 
\end{Lemma}

\begin{proof}
This follows from Corollary \ref{normneg} similar to the computation in \cite[Theorem 4.10]{ground}.

\end{proof}

It will be useful to consider the  following bidistribution, $K_A$ that satisfies  $\partial_{t}K_A=K_G$.
\begin{Corollary}
Let $K_A\in \mathcal D'(M\times M)$  be the bidistribution given by 
\begin{equation*}
K_A(u\otimes v):=\int_{M}\left(\int_{-\infty}^{\infty}\sum_j \lambda_j^{-2} \cos(\lambda_j(t-s))\phi_j(x)\int_\Sigma\phi_j(y)v(s,y)\sqrt{h(y)}dyds\right)u(t,x)\sqrt{h(x)}dxdt,
\end{equation*}

Then,
\begin{align}\label{normcor}
K_A\in H_{loc}^{\frac{1}{2}-\epsilon}(M\times M) \text{ for every } \epsilon >0.
\end{align}
\end{Corollary}

\begin{proof}
This follows from Proposition \ref{normeigen} similar to the computation in \cite[Corollary 4.11]{ground}.
\begin{comment}
Chose $\psi$ as in the proof of Lemma \ref{normkg}. 
Direct computation shows that  
%
\begin{align*}
&\Big|{\mathcal F_{(t,s)\to(\xi_0,\eta_0)}}
\Big(\frac{\psi(t,s)}{\lambda_j^2}\cos({\lambda_j(t-s)})\Big)(\xi_0, \eta_0)\Big|^2\\
&\le \frac{1}{\lambda_j^{4}}\left(\frac{C}{\ang{\xi_0-\lambda_j}^N\ang{\eta_0+\lambda_j}^N}+ \frac{C}{\ang{\xi_0+\lambda_j}^N\ang{\eta_0-\lambda_j}^N}\right)^2
\end{align*}
%
for all $N\in \mathbb{N}$ and a suitable constant $C$.
This gives for $s\ge0$
%
\begin{align*}
&\|\psi K_A\|_{H^{s}(M\times M)}\\
&=\sum_{j=k}\int_{\mathbb{R}^2}(|\xi_0|^2+|\eta_0|^2+\lambda_j^2+\lambda_k^2)^{s}
\Big|{\mathcal F_{(t,s)\to(\xi_0,\eta_0)}}
\Big(\frac{\psi(t,s)}{\lambda_j^2}\cos({\lambda_j(t-s)})\Big)(\xi_0, \eta_0)\Big|^2d\xi_0d\eta_0\\
&\le C\sum_{j=k}\frac{(\lambda_j^2+\lambda_k^2)^{s}}{\lambda_j^{4}}
\le C' \sum_j{\lambda_j^{2s-4}}
\end{align*}
 for a suitable constant $C'$.% (that may be changing from line to line).

According to Weyl's law for non-smooth metrics \cite[Theorem 1.1]{zielinski} we have the estimate $l^{\frac{2}{3}}\le C \lambda_l^2$ for a suitable constant $C$. 
This gives  for $s=\frac{1}{2}-\epsilon$, $0<\epsilon\le \frac{1}{2}$.
%
\begin{align}
\|\psi K_A\|^2_{H^{s}(M\times M)}&\le \sum_l\frac{C'}{\lambda_l^{3+2\epsilon}}\le \sum_l\frac{C''}{l^{1+\frac23 \epsilon}}<\infty
\end{align}
for a suitable constant $C''$. % (that may be changing from line to line).
\end{comment}
\end{proof}

\subsubsection{Wavefront Set Estimates}

{Now we show some some helpful lemmas in order to prove Theorem \ref{main} and Theorem \ref{main3} which are the main results of the section.}

First, we establish the microlocal  regularity of $K_G$ outside the set $\Char(P)\times\Char(P)$.

{In the following proofs, we use the distribution $K_A$, because a direct application of Theorem \ref{elliptic} for $K_G$ is not possible, since for $\delta$ close to $1$ the above $\sigma$ cannot take the value  $-\frac{1}{2}$.}

\begin{Lemma}\label{char}
For $\tau>2$ and any $\tilde{\epsilon}>0$ 
\begin{equation}
WF^{-\frac{1}{2}-\tilde{\epsilon}+\tau}(K_G)\subset \Char(P)\times\Char(P).
\end{equation} 
\end{Lemma}

\begin{proof}
This is an application of Theorem \ref{elliptic}, the observation that $K_A$ satisfies $(\partial_t+\partial_s)K_A=0$ and
$WF^{-\frac{1}{2}-\tilde{\epsilon}+\tau}(K_G)\subset WF^{\frac{1}{2}-\tilde{\epsilon}+\tau}(K_A)$.
The proof is along the lines \cite[Lemma 4.13]{ground}
\end{proof}

Now we establish that points above the diagonal are of a specific form.

\begin{Lemma}\label{diag3}
 If $(\tx,\txi,\tx,\teta)\in WF^{-\frac{3}{2}-\tilde{\epsilon}+\tau}(K_G)$ for {$\tau>2$} and some $\tilde{\epsilon}>0$, then $\teta=-\txi$.
\end{Lemma}

\begin{proof}
This is a consequence of Theorem \ref{taylor} combined with the support properties of $K_G$. The proof is along the lines of that for \cite[Lemma 4.16]{ground}

\end{proof}

Now we state one of the main results:

\begin{Theorem}\label{main}
Let $(M,g)$ be a $C^{\tau}$ ultrastatic spacetime with $\tau>2$ and $K_G$ the causal propagator. Then {$WF'^{-\frac{3}{2}-{\epsilon}+\tau}(K_G)\subset C$ for
every ${\epsilon}>0$} and $C$ as in Eq.\eqref{C}. % when is wavefrontset empty?
\end{Theorem}

\begin{proof}

Let $(\tx,\txi,\ty,-\teta)\in WF^{-\frac{3}{2}-{\epsilon}+\tau}(K_{G})$.
The propagation of singularities result (Theorem \ref{propagation})  implies that  $(\gamma(\tx,\txi),\gamma (\ty,-\teta))\in { WF^{-\frac{1}{2}-\epsilon+\tau}}(K_A)$, where  $\gamma(\tx,\txi)$ is the null bicharacteristic with initial data $(\tx,\txi)$ and $\gamma(\ty,-\teta)$ is the null bicharacteristic with initial data $(\ty,-\teta)$. As a consequence of Lemma \ref{char}, Lemma 6.4 , the fact that $(\partial_t+\partial_s) K_G=0$ and the inclusion $WF^{s}(K_A)\subset WF^{s-1}(K_G)\cup \Char(\partial_t)$ for all $s\in\mathbb{R}$, we have  $(\gamma(\tx,\txi),\gamma (\ty,-\teta))\in {W F^{-\frac{3}{2}-\epsilon+\tau}}(K_G)$. Then we can apply Theorem \ref{taylor} combined with Lemma \ref{diag3} to obtain the result.
The proof is along the lines \cite[Theorem 4.17]{ground}.

\end{proof}

For the analysis of adiabatic states it is enough to work with the inclusion shown above. However, in the smooth case we have an equality of sets. 
In Theorem \ref{main3}, we show that this equality holds under stronger regularity assumptions on the metric.

First we show the following lemma

\begin{Lemma}\label{diag1}
Let $(\tx,\txi)\in \Char(P)$ with $P$ as in Eq. \eqref{kg}.Then   $(\tx,\tx,\txi,-\txi)\in WF^{\frac{3}{2}+{\epsilon}}(K_G)$ for all ${\epsilon}>0.$%\yafet{corrections}
\end{Lemma}

\begin{proof}
Since $WF^{s_1}\subset WF^{s_2}$ for $s_1\le s_2$, it is enough to show the result for small $\epsilon$.
Let  $Q:=\R\times\Sigma^2$. We define the embedding $f:Q\rightarrow M\times M$ by $f(s,x,y)=(s,x,s,y)$.  The set of normals of the map $f$ is 
\begin{align*}
 N_{f}&=\{(f(s,x,y),\txi,\teta)\in T^*(M\times M); ^tf'(s,x,y)(\txi,\teta)=0\}\\
&=\{(s,x,s,y,\xi^0,0,-\xi^0, 0)\in T^*(M\times M)\},
\end{align*}
where $^tf'$ is the transpose of the differential of $f$. 
{ In particular, $N_f\cap (\Char P\times \Char P) = \emptyset$.} 
 By Lemma \ref{char}, %below,
\begin{equation*}
WF^{\frac{3}{2}+{\epsilon}}(K_G)\cap N_{f}=\emptyset
\end{equation*}

and therefore 

\begin{equation*}
WF^{\frac{1}{2}+\epsilon}(\partial_t K_G)\cap N_f\subset  WF^{\frac{3}{2}+{\epsilon}}(K_G)\cap N_{f}=\emptyset
\end{equation*}

for suitably small $\epsilon>0$.
%Since $K_G= \partial_tK_A$, \eqref{char2} shows that 
%%
%\begin{equation*}\red{
%\red{WF^{\frac{1}{2}+\epsilon}(\partial_t K_G)\subset WF^{\frac{3}{2}+\epsilon}(K_G)\subset  \Char(P)\times\Char(P)}}
%\end{equation*}
%%
Therefore Proposition B.7 from \cite{adiabatic} implies that the restriction of   
$\partial_tK_G$ to ${Q}$ is defined and satisfies 
\begin{align}\label{diagonal}
WF^{\epsilon}(\partial_tK_G|_{Q})&\subset f^*(WF^{{\frac12+\epsilon}}(\partial_tK_G))\\
&=\{(s,x,y, ^tf'(\txi,\teta))\in T^*Q; (f(s,x,y),\txi,\teta)\in {WF^{\frac12+\epsilon}}(\partial_t K_G)\}\nonumber.
\end{align} 
As a distribution,  $\partial_t{K_G}|_Q$ is  given by 
$$\partial_tK_G|_{Q}(s,x,y)=-\sum_j \phi_j(x)\phi_j(y),$$
i.e.,  it acts on the {non-smooth} density  $\psi_1(s)\psi_2(x)\psi_3(y)\sqrt{h}(x)\sqrt{h}(y)dxdy$,   by  
\begin{equation}
\langle \partial_t K_G|_{Q},\psi_1\psi_2\psi_3\rangle=-\int_{-\infty}^\infty \psi_1(p)dp\int_\Sigma\psi_2(w)\psi_3(w)\sqrt{h(w)}dw.
\end{equation}
Therefore its Fourier transform is given by 
\begin{equation}
({\cal{F}}(\partial_tK_G|_{Q}))({\chi},\xi,\eta)= %\int_{-r}^re^{-i{\chi}p}dp
\delta_0(\chi)\otimes \int_\Sigma e^{-i w(\xi+\eta)}\sqrt{h(w)}dw.
\end{equation}

{Moreover, we have $(\partial_tK_G|_{Q}-1\otimes\delta(x-y))(\psi)=0$ for all smooth densities on $\mathbb{R}\times\Sigma\times \Sigma$. Therefore $\partial_tK_G|_{Q}=1\otimes\delta(x-y)$ as elements of ${\cal{D}}'(\mathbb{R}\times \Sigma\times \Sigma)$.  This implies

$$
WF^s(\partial_tK_G|_{Q})=
\begin{cases}
\emptyset, s<-\frac{3}{2}\\
(s,x,x,0,\xi,-\xi) \text{ for all } \xi\in T^*_{x}\Sigma, s\ge -\frac{3}{2}.
\end{cases}
$$

}

%
%\sout{Taking $\chi=0$ and $-\xi=\eta$ we see that $({\cal{F}}(\partial_tK_G|_{Q}))(0,\xi,-\xi)$ does not decay % and since $-2-\epsilon+\tau\delta>-4-\epsilon$, 
%and therefore $(s,x,x,0,\xi,-\xi)\in WF^{\epsilon }(\partial_tK_G|_{Q})$.}

Using Eq. (\ref{diagonal}) we find that there exists $\xi_0$ such that  $(s,x,s,x,\xi^0,\xi, -\xi^0,-\xi)\in WF^{\frac12+\epsilon }(\partial_t K_G)$ for each $\xi \in T^*\Sigma$.  

 According to Proposition B.3 from \cite{adiabatic} 
\begin{equation}
WF^{\frac{1}{2}+\epsilon }(\partial_t K_G)\subset WF^{\frac{3}{2}+\epsilon }( K_G).
\end{equation}

Since the wavefront set is contained in $\Char(P)\times\Char(P)$ we obtain from Lemma \ref{char}    $(s,x,s,x,\xi^0,\xi, -\xi^0,-\xi)\in\Char(P)\times\Char(P)$ with $\xi_0^2=h^{ij}\xi_i\xi_j$. Without loss of generality we choose a sign for $\xi_0$ i.e. $\xi_0:=\sqrt{h^{ij}\xi_i\xi_j}$.

Now we show that if $(s,x,s,x,\xi^0,\xi, -\xi^0,-\xi)\in WF^{\frac{3}{2}+\epsilon }( K_G)$ then  $(s,x,s,x,-\xi^0,-\xi, \xi^0,\xi)\in WF^{\frac{3}{2}+\epsilon }( K_G)$.
{
The diffeomorphism $f_1(t,x,s,y)=(s,y,t,x)$ { has  the set of normals $N_{f_1}=\{(s,y,t,x,0,0,0,0)\in T^*(M\times M)\}$ } which has empty intersection with $WF(K_G)$. Then \cite[Theorem 8.2.3]{hormander1} and the invariance of the Sobolev wavefront set implies that

\begin{equation}
WF^{\frac{3}{2}+\epsilon }( f_1^* K_G)=f_1^*WF^{\frac{3}{2}+\epsilon }( K_G).
\end{equation}

Moreover $f_1^*K_G=-K_G$ which gives 

\begin{equation}\label{pulltime}
WF^{\frac{3}{2}+\epsilon }( K_G)=f_1^*WF^{\frac{3}{2}+\epsilon }( K_G).
\end{equation}

Now since  $(s,x,s,x,\xi^0,\xi, -\xi^0,-\xi)\in WF^{\frac{3}{2}+\epsilon }(K_G)$ then  we have  $(s,x,s,x,-\xi^0,-\xi, \xi^0,\xi)\in WF^{\frac{3}{2}+\epsilon }(K_G)$ by Eq.\eqref{pulltime}.

Notice that we also  have to show that $(s,x,s,x,-\xi^0,\xi, \xi^0,-\xi)$ and $(s,x,s,x,\xi^0,-\xi, -\xi^0,\xi)$ are in $ WF^{\frac{3}{2}+\epsilon }(K_G)$.

In this case we use the diffeomorphism $f_2(t,x,s,y)=(s,x,t,y)$ {that has  the set of normals $N_{f_2}=\{(s,x,t,y,0,0,0,0)\in T^*(M\times M)\}$ } which has empty intersection with $WF(K_G)$. Then \cite[Theorem 8.2.3]{hormander1} and the invariance of the Sobolev wavefront set implies that

\begin{equation}
WF^{\frac{3}{2}+\epsilon }( f_2^* K_G)=f_2^*WF^{\frac{3}{2}+\epsilon }( K_G).
\end{equation}

Moreover $f_2^*K_G=-K_G$ which gives 

\begin{equation}\label{pulltime2}
WF^{\frac{3}{2}+\epsilon }( K_G)=f_2^*WF^{\frac{3}{2}+\epsilon }( K_G).
\end{equation}

Now since  $(s,x,s,x,\xi^0,\xi, -\xi^0,-\xi)\in WF^{\frac{3}{2}+\epsilon }(K_G)$ then we have  $(s,x,s,x,-\xi^0,\xi, \xi^0,-\xi)\in WF^{\frac{3}{2}+\epsilon }(K_G)$ by Eq.\eqref{pulltime2}. Using $f_1$ we obtain  $(s,x,s,x,\xi^0,-\xi, -\xi^0,\xi)\in WF^{\frac{3}{2}+\epsilon }(K_G)$.  This gives the desired result. }

\end{proof}

Now we show the equality of sets as in the smooth case. 

\begin{Theorem}\label{main3}
{Let $(M,g)$ be a $C^{\tau}$ ultrastatic spacetime with $\tau>3$ and $K_G$ the causal propagator. Then  $C\subset WF'^{-\frac{3}{2}+\tau-\tep}(K_G)$ for all $\tep<\tau-3$ and $C$ as in Eq.\eqref{C}. In particular, we have  $C\subset WF'^{s}(K_G)$ for all $s>\frac{3}{2}$.}
\end{Theorem}

\begin{proof}
{
Under the additional regularity assumption and arguing locally as in Theorem \ref{elliptic}, we have $P^b_{(t,x)}K_A,P^b_{(s,y)}K_A\in H^{\frac{3}{2}+\tep}(M\times M)$} and therefore for $(\tx,\txi,\ty,\teta)\in WF^{\frac{3}{2}+\tilde{\epsilon}}(K_G)\subset WF^{\frac{5}{2}+\tilde{\epsilon}}(K_A) $ we can choose $s=\frac{3}{2}+\tilde{\epsilon}$ in Theorem \ref{propagation}. 

Now if $(\txx,\txixi)=(\tx,\txi,\ty,-\teta)\in C'$  then  there is a null geodesic $\gamma$  such that $\gamma(t_1)=\tx,\gamma(t_2)=\ty$  and $g(\cdot, {\dot{\gamma}})|_{T_{\tx}M}=\txi, g(\cdot, {\dot{\gamma}})|_{T_{\ty}M}=\teta$. Now, $(\tx,\txi,\tx,-\txi)\in C'$ and by Lemma \ref{diag1} $(\tx,\txi,\tx,-\txi)\in WF^{\frac{3}{2}+\epsilon}(K_G)$ for $\epsilon>0$ which implies for $\tilde{\epsilon}<\tau-3$ that {$(\tx,\txi,\tx,-\txi)\in WF^{-\frac{3}{2}+\tau-\tilde{\epsilon}}(K_G)\subset  WF^{-\frac{1}{2}+\tau-\tilde{\epsilon}}(K_A)$. Applying Theorem \ref{propagation} to $P_{(t,x)}K_A,P_{(s,y)}K_A$ with the $s$ described above we have 
$(\gamma(\tx,\txi),\gamma(\tx,-\txi))\in WF^{-\frac{1}{2}+\tau-\tep}(K_A)$. Using the same argument as in Theorem \ref{main}  this implies $(\tx,\txi,\ty,-\teta)=(\txx,\txixi)\in WF^{-\frac{3}{2}+\tau-\tep}(K_G)$.}
\end{proof}

{
\begin{rem}
The combination of Theorem \ref{main} with Theorem \ref{main3} gives $$WF'^{-\frac{3}{2}+\tau-\tilde{\epsilon}}(K_G)=C$$
for $\tau>3$ and {$\tep<\tau-3$.}
\end{rem}
}

\subsubsection{The $C^{1,1}$ Case}
The following theorem states the result for the case of $C^{1,1}$ regularity.
{
\begin{Theorem}\label{main2}
Let $(M,g)$ be a $C^{1,1}$ ultrastatic spacetime and $K_G$ the causal propagator. Then $WF'^{\frac{1}{2}-\tilde{\epsilon}}(K_G)\subset C$ for all $\tep>0$. % 
\end{Theorem}
}
{\em Proof of Theorem} \ref{main2}. 
In order to show the theorem we will state how different results of the paper change under this regularity. 

From the comment above Theorem \ref{propagation} we know that  Theorem \ref{propagation} still holds. {Notice that $C^{1,1}\subset C_*^2$ \cite[Chapter 1, Eq.(1.21)]{tools}}. 

Also, notice that a $C^{1,1}$ metric guarantees the existence and uniqueness of the Hamiltonian flow which is critical for the proof. 
  Theorem \ref{elliptic} holds even for $\tau>1$.

Lemma \ref{causal} requires no modification, since the results on global hyperbolicity still hold for this regularity \cite[Corollary 3.4]{clemens}. The hypothesis in \cite[Theorem 1.1]{zielinski} is the requirement that the coefficients of the principal part have one derivative that is Lipschitz which is clearly  satisfied in the $C^{1,1}$ case. Hence  Lemma \ref {normkg} holds.

For Lemma \ref{diag3} and Theorem \ref{main} the only thing to notice is that in this case $P_{(t,x)}^b K_A,P_{(s,y)}^b K_A\in H^{\frac12-\tep}(M\times M)$ (arguing locally as in Theorem \Ref{elliptic}) and therefore we can apply Theorem \ref{propagation} for $s= \frac{1}{2}-\tilde{\epsilon}$. In this section we 
have applied the version of Theorem \ref{propagation} after \cite[Proposition 11.4]{tools}.
\qed

%-------------------------------------------------------------------------------------------------------------------------------------------------------
% ACKNOWLEDGEMENTS SECTION<
%-------------------------------------------------------------------------------------------------------------------------------------------------------
\newpage
{\bf{Acknowledgements}}

We are grateful to  Chris Fewster, Bernard Kay and James Vickers for helpful discussions. We also thank the referees for valuable comments and suggestions.  \\

{\bf{Data Availability Statement}}

Data sharing not applicable to this article as no datasets were generated or analysed during the current study.

{\bf{Funding}}

The work of YESS has been partially funded by Next Generation EU through the project “Geometrical and Topological effects on Quantum Matter (GeTOnQuaM)”. The research activities of YESS have been carried out in the framework of the INFN Research Project QGSKY.

{\bf{ Conflicts of interests/Competing interests}}

The authors have no relevant financial or non-financial interests to disclose.
%-------------------------------------------------------------------------------------------------------------------------------------------------------
% BIBLIOGRAPHY SECTION
%-------------------------------------------------------------------------------------------------------------------------------------------------------
\bibliographystyle{plain}
\bibliography{biblio}

\begin{thebibliography}{10}

\bibitem{carolina}
Helmut Abels and Carolina Neira~Jim{\'e}nez.
\newblock Nonsmooth pseudodifferential boundary value problems on manifolds.
\newblock {\em Journal of Pseudo-Differential Operators and Applications},
  10:415--453, 2019.

\bibitem{adler}
R.~J. Adler, J.~D. Bjorken, P.~Chen, and J.~S. Liu.
\newblock Simple analytical models of gravitational collapse.
\newblock {\em American Journal of Physics}, 73(12):1148--1159, 2005.

\bibitem{amann}
Herbert Amann.
\newblock {\em Linear and quasilinear parabolic problems. {V}ol. {I}},
  volume~89 of {\em Monographs in Mathematics}.
\newblock Birkh\"{a}user/Springer, Cham, 1995.

\bibitem{lashi}
Lashi Bandara.
\newblock Rough metrics on manifolds and quadratic estimates.
\newblock {\em Math. Z.}, 283(3-4):1245--1281, 2016.

\bibitem{bgp}
Christian B\"{a}r, Nicolas Ginoux, and Frank Pf\"{a}ffle.
\newblock {\em Wave equations on {L}orentzian manifolds and quantization}.
\newblock ESI Lectures in Mathematics and Physics. European Mathematical
  Society (EMS), Z\"{u}rich, 2007.

\bibitem{beals}
Michael Beals and Michael Reed.
\newblock Propagation of singularities for hyperbolic pseudodifferential
  operators with nonsmooth coefficients.
\newblock {\em Comm. Pure Appl. Math.}, 35(2):169--184, 1982.

\bibitem{ben}
Marco Benini, Alexander Schenkel, and Lukas Woike.
\newblock Homotopy theory of algebraic quantum field theories.
\newblock {\em Letters in Mathematical Physics}, 109:1487--1532, 2019.

\bibitem{sanchez}
Antonio~N. Bernal and Miguel S\'{a}nchez.
\newblock Globally hyperbolic spacetimes can be defined as `causal' instead of
  `strongly causal'.
\newblock {\em Classical Quantum Gravity}, 24(3):745--749, 2007.

\bibitem{bernalsplitting}
Antonio~N. Bernal and Miguel Sánchez.
\newblock Smoothness of time functions and the metric splitting of globally
  hyperbolic spacetimes.
\newblock {\em Communications in Mathematical Physics}, 257:43--50, 2005.

\bibitem{bony}
Jean-Michel Bony.
\newblock Calcul symbolique et propagation des singularit\'{e}s pour les
  \'{e}quations aux d\'{e}riv\'{e}es partielles non lin\'{e}aires.
\newblock {\em Ann. Sci. \'{E}cole Norm. Sup. (4)}, 14(2):209--246, 1981.

\bibitem{brun}
Romeo Brunetti, Klaus Fredenhagen, and Kasia Rejzner.
\newblock Locally covariant approach to effective quantum gravity.
\newblock In C.~Bambi, L.~Modesto, and I.L. Shapiro, editors, {\em Handbook of
  Quantum Gravity}. Springer Singapore, 2023.
\newblock Section: Perturbative Quantum Gravity.

\bibitem{brun1}
Romeo Brunetti, Klaus Fredenhagen, and Rainer Verch.
\newblock The generally covariant locality principle – a new paradigm for
  local quantum field theory.
\newblock {\em Communications in Mathematical Physics}, 237(1-2):31--68, 2003.

\bibitem{buch2}
Detlev Buchholz, Fabio Ciolli, Giuseppe Ruzzi, and Ezio Vasselli.
\newblock {\em Gauss’s Law, the Manifestations of Gauge Fields, and Their
  Impact on Local Observables}, pages 71--92.
\newblock Springer, City, 2023.

\bibitem{buch}
Detlev Buchholz and Klaus Fredenhagen.
\newblock A c*-algebraic approach to interacting quantum field theories.
\newblock {\em Communications in Mathematical Physics}, 377:947--969, 2020.

\bibitem{capo}
Matteo Capoferri and Simone Murro.
\newblock Global and microlocal aspects of dirac operators: propagators and
  hadamard states.
\newblock {\em Advances in Differential Equations}, September 2023.

\bibitem{fluids}
Demetrios Christodoulou.
\newblock Self-gravitating relativistic fluids: a two-phase model.
\newblock {\em Arch. Rational Mech. Anal.}, 130(4):343--400, 1995.

\bibitem{cr}
Demetrios Christodoulou.
\newblock {\em The formation of black holes in general relativity}.
\newblock EMS Monographs in Mathematics. European Mathematical Society (EMS),
  Z\"{u}rich, 2009.

\bibitem{vien}
Piotr~T. Chru\'{s}ciel and James D.~E. Grant.
\newblock On {L}orentzian causality with continuous metrics.
\newblock {\em Classical Quantum Gravity}, 29(14):145001, 32, 2012.

\bibitem{dafermos}
Mihalis Dafermos.
\newblock Stability and instability of the {C}auchy horizon for the spherically
  symmetric {E}instein-{M}axwell-scalar field equations.
\newblock {\em Ann. of Math. (2)}, 158(3):875--928, 2003.

\bibitem{moretti}
Claudio Dappiaggi, Valter Moretti, and Nicola Pinamonti.
\newblock Rigorous construction and {H}adamard property of the {U}nruh state in
  {S}chwarzschild spacetime.
\newblock {\em Adv. Theor. Math. Phys.}, 15(2):355--447, 2011.

\bibitem{integral}
Giuseppe De~Cecco and Giuliana Palmieri.
\newblock Integral distance on a {L}ipschitz {R}iemannian manifold.
\newblock {\em Math. Z.}, 207(2):223--243, 1991.

\bibitem{der}
Jan Derezi\'{n}ski and Daniel Siemssen.
\newblock Feynman propagators on static spacetimes.
\newblock {\em Rev. Math. Phys.}, 30(3):1850006, 23, 2018.

\bibitem{dim}
J.~Dimock.
\newblock Algebras of local observables on a manifold.
\newblock {\em Communications in Mathematical Physics}, 77(3):219--228, 1980.

\bibitem{drago}
Nicoló Drago, Nicolas Ginoux, and Simone Murro.
\newblock Møller operators and hadamard states for dirac fields with mit
  boundary conditions.
\newblock {\em Documenta Mathematica}, 27:1693--1737, 2022.

\bibitem{duistermaat}
J.~J. Duistermaat and L.~H\"{o}rmander.
\newblock Fourier integral operators. {II}.
\newblock {\em Acta Math.}, 128(3-4):183--269, 1972.

\bibitem{few2}
Christopher~J. Fewster.
\newblock Locally covariant quantum field theory and the problem of formulating
  the same physics in all space–times.
\newblock {\em Philosophical Transactions of the Royal Society A: Mathematical,
  Physical and Engineering Sciences}, 2015, 2015.

\bibitem{fewster}
Christopher~J. Fewster and Rainer Verch.
\newblock The necessity of the {H}adamard condition.
\newblock {\em Classical Quantum Gravity}, 30(23):235027, 20, 2013.

\bibitem{fulling}
S.~A. Fulling, F.~J. Narcowich, and Robert~M. Wald.
\newblock Singularity structure of the two-point function in quantum field
  theory in curved spacetime. {II}.
\newblock {\em Ann. Physics}, 136(2):243--272, 1981.

\bibitem{wrochna}
Christian G\'{e}rard and Michal Wrochna.
\newblock Construction of {H}adamard states by pseudo-differential calculus.
\newblock {\em Comm. Math. Phys.}, 325(2):713--755, 2014.

\bibitem{GS13}
Nadine Gro{\ss}e and Cornelia Schneider.
\newblock Sobolev spaces on {R}iemannian manifolds with bounded geometry:
  general coordinates and traces.
\newblock {\em Math. Nachr.}, 286(16):1586--1613, 2013.

\bibitem{gerdir}
Christian Gérard and Théo Stoskopf.
\newblock Hadamard states for quantized dirac fields on lorentzian manifolds of
  bounded geometry.
\newblock {\em arXiv}, June 2021.

\bibitem{haag}
Rudolf Haag and Daniel Kastler.
\newblock An algebraic approach to quantum field theory.
\newblock {\em Journal of Mathematical Physics}, 5(7):848--861, 1964.

\bibitem{holdir}
S.~Hollands.
\newblock The hadamard condition for dirac fields and adiabatic states on
  robertson-walker spacetimes.
\newblock {\em Communications in Mathematical Physics}, 216:635--661, 2001.

\bibitem{hollands}
Stefan Hollands and Robert~M. Wald.
\newblock Local {W}ick polynomials and time ordered products of quantum fields
  in curved spacetime.
\newblock {\em Comm. Math. Phys.}, 223(2):289--326, 2001.

\bibitem{hormander1}
Lars H{\"o}rmander.
\newblock {\em The Analysis of Linear Partial Differential Operators I:
  Distribution Theory and Fourier Analysis}.
\newblock Classics in Mathematics. Springer, Berlin; New York, 2003.
\newblock Reprint of 2nd ed., c1990-c1994.

\bibitem{green}
Günther H\"{o}rmann, Yafet Sanchez~Sanchez, Christian Spreitzer, and James~A.
  Vickers.
\newblock Green operators in low regularity spacetimes and quantum field
  theory.
\newblock {\em Classical Quantum Gravity}, 37(17):175009, 50, 2020.

\bibitem{junker}
Wolfgang Junker.
\newblock Hadamard states, adiabatic vacua and the construction of physical
  states for scalar quantum fields on curved spacetime.
\newblock {\em Rev. Math. Phys.}, 8(8):1091--1159, 1996.

\bibitem{adiabatic}
Wolfgang Junker and Elmar Schrohe.
\newblock Adiabatic vacuum states on general spacetime manifolds: definition,
  construction, and physical properties.
\newblock {\em Ann. Henri Poincar\'{e}}, 3(6):1113--1181, 2002.

\bibitem{kaywald}
Bernard~S. Kay and Robert~M. Wald.
\newblock Theorems on the uniqueness and thermal properties of stationary,
  nonsingular, quasifree states on spacetimes with a bifurcate {K}illing
  horizon.
\newblock {\em Phys. Rep.}, 207(2):49--136, 1991.

\bibitem{kav}
Igor Khavkine and Valter Moretti.
\newblock Algebraic {QFT} in curved spacetime and quasifree {H}adamard states:
  an introduction.
\newblock In {\em Advances in algebraic quantum field theory}, Math. Phys.
  Stud., pages 191--251. Springer, Cham, 2015.

\bibitem{l2}
Sergiu Klainerman, Igor Rodnianski, and Jeremie Szeftel.
\newblock The bounded {$L^2$} curvature conjecture.
\newblock {\em Invent. Math.}, 202(1):91--216, 2015.

\bibitem{kunz}
Michael Kunzinger, Roland Steinbauer, and Milena Stojkovi\'{c}.
\newblock The exponential map of a {$C^{1,1}$}-metric.
\newblock {\em Differential Geom. Appl.}, 34:14--24, 2014.

\bibitem{spin}
H.~Blaine Lawson, Jr. and Marie-Louise Michelsohn.
\newblock {\em Spin geometry}, volume~38 of {\em Princeton Mathematical
  Series}.
\newblock Princeton University Press, Princeton, NJ, 1989.

\bibitem{longo}
Roberto Longo and Karl-Henning Rehren.
\newblock Local fields in boundary conformal {QFT}.
\newblock {\em Reviews in Mathematical Physics}, 16(07):909--960, 2004.

\bibitem{mar}
J\"{u}rgen Marschall.
\newblock Pseudodifferential operators with nonregular symbols of the class
  {$S^m_{\rho\delta}$}.
\newblock {\em Comm. Partial Differential Equations}, 12(8):921--965, 1987.

\bibitem{minguzzi}
Ettore Minguzzi.
\newblock Causality theory for closed cone structures with applications.
\newblock {\em Rev. Math. Phys.}, 31(5):1930001, 139, 2019.

\bibitem{stars}
G.~Miniutti, J.~A. Pons, E.~Berti, L.~Gualtieri, and V.~Ferrari.
\newblock Non-radial oscillation modes as a probe of density discontinuities in
  neutron stars.
\newblock {\em Monthly Notices of the Royal Astronomical Society},
  338(2):389--400, 2003.

\bibitem{muller}
Olaf Müller.
\newblock Asymptotic flexibility of globally hyperbolic manifolds.
\newblock {\em Comptes Rendus Mathematique}, 350(7-8):421--423, 2012.

\bibitem{putman}
C.~R. Putnam.
\newblock {\em Commutation properties of {H}ilbert space operators and related
  topics}.
\newblock Ergebnisse der Mathematik und ihrer Grenzgebiete, Band 36.
  Springer-Verlag New York, Inc., New York, 1967.

\bibitem{rad}
Marek~J. Radzikowski.
\newblock Micro-local approach to the {H}adamard condition in quantum field
  theory on curved space-time.
\newblock {\em Comm. Math. Phys.}, 179(3):529--553, 1996.

\bibitem{reed}
Michael Reed and Barry Simon.
\newblock {\em Methods of modern mathematical physics. {II}. {F}ourier
  analysis, self-adjointness}.
\newblock Academic Press [Harcourt Brace Jovanovich, Publishers], New
  York-London, 1975.

\bibitem{reed3}
Michael~C. Reed.
\newblock On self-adjointness in infinite tensor product spaces.
\newblock {\em J. Functional Analysis}, 5:94--124, 1970.

\bibitem{clemens}
Clemens S\"{a}mann.
\newblock Global hyperbolicity for spacetimes with continuous metrics.
\newblock {\em Ann. Henri Poincar\'{e}}, 17(6):1429--1455, 2016.

\bibitem{ground}
Yafet~Sanchez Sanchez and Elmar Schrohe.
\newblock Adiabatic ground states in non-smooth spacetimes.
\newblock {\em Ann. Henri Poinc.}, 24, 2023.

\bibitem{ko}
Ko~Sanders.
\newblock Equivalence of the (generalised) {H}adamard and microlocal spectrum
  condition for (generalised) free fields in curved spacetime.
\newblock {\em Comm. Math. Phys.}, 295(2):485--501, 2010.

\bibitem{smith}
Hart~F. Smith.
\newblock A parametrix construction for wave equations with {$C^{1,1}$}
  coefficients.
\newblock {\em Ann. Inst. Fourier (Grenoble)}, 48(3):797--835, 1998.

\bibitem{str}
Robert~S. Strichartz.
\newblock Analysis of the {L}aplacian on the complete {R}iemannian manifold.
\newblock {\em J. Functional Analysis}, 52(1):48--79, 1983.

\bibitem{stro}
A.~Strohmaier.
\newblock The {R}eeh-{S}chlieder property for quantum fields on stationary
  spacetimes.
\newblock {\em Comm. Math. Phys.}, 215(1):105--118, 2000.

\bibitem{szeftelparametrix}
J\'{e}r\'{e}mie Szeftel.
\newblock Parametrix for wave equations on a rough background. {I}:
  {R}egularity of the phase at initial time. {II}: {C}onstruction and control
  at initial time.
\newblock {\em Ast\'{e}risque}, 443:ix+275, 2023.

\bibitem{tataru}
Daniel Tataru.
\newblock Strichartz estimates for operators with nonsmooth coefficients and
  the nonlinear wave equation.
\newblock {\em Amer. J. Math.}, 122(2):349--376, 2000.

\bibitem{taylor}
Michael~E. Taylor.
\newblock {\em Pseudodifferential operators and nonlinear {PDE}}, volume 100 of
  {\em Progress in Mathematics}.
\newblock Birkh\"{a}user Boston, Inc., Boston, MA, 1991.

\bibitem{tools}
Michael~E. Taylor.
\newblock {\em Tools for {PDE}}, volume~81 of {\em Mathematical Surveys and
  Monographs}.
\newblock American Mathematical Society, Providence, RI, 2000.
\newblock Pseudodifferential operators, paradifferential operators, and layer
  potentials.

\bibitem{waldten}
Robert~M. Wald.
\newblock The back reaction effect in particle creation in curved spacetime.
\newblock {\em Communications in Mathematical Physics}, 54:1--19, 1977.

\bibitem{wald}
Robert~M. Wald.
\newblock {\em Quantum field theory in curved spacetime and black hole
  thermodynamics}.
\newblock Chicago Lectures in Physics. University of Chicago Press, Chicago,
  IL, 1994.

\bibitem{waters}
Alden Waters.
\newblock A parametrix construction for the wave equation with low regularity
  coefficients using a frame of {G}aussians.
\newblock {\em Commun. Math. Sci.}, 9(1):225--254, 2011.

\bibitem{zielinski}
Lech Zielinski.
\newblock Sharp spectral asymptotics and {W}eyl formula for elliptic operators
  with non-smooth coefficients. {II}.
\newblock {\em Colloq. Math.}, 92(1):1--18, 2002.

\end{thebibliography}

%%%%%%%%%%%%%%%%%%%%%%%%%%%%%%% old biliography %%%%%%%%%%%%%%%%%%%%%%%

\Addresses

%--------------------------------------------------------------- END DOCUMENT --%------------------------------------------------------------
\end{document}